\documentclass[11pt,a4paper]{article}
\usepackage{amssymb,amsmath,amsfonts,amsthm,hyperref,here,enumerate,xfrac,mathtools,graphicx,pstool,xcolor,comment,authblk}
\usepackage{framed,tikz}
\usepackage{tkz-euclide}
\usepackage{mlmodern}

\def\le{\leqslant}
\def\leq{\leqslant}
\def\ge{\geqslant}
\def\geq{\geqslant}
\newcommand\mycom[2]{\genfrac{}{}{0pt}{}{#1}{#2}}

\usetikzlibrary{arrows}
\usetikzlibrary{arrows.meta}
\usetikzlibrary{patterns,decorations.pathreplacing}
\usepackage[font={small},width=1.1\textwidth]{caption}

\definecolor{light}{rgb}{.9,.9,.9}
\definecolor{darkgreen}{rgb}{0,.6,0}
\numberwithin{equation}{section}

\newcommand{\R}{\mathbb R}
\newcommand{\N}{\mathbb N}
\newcommand{\Z}{\mathbb Z}

\newcommand{\Q}{\mathbb Q}
\newtheoremstyle{plain}
  {10pt}
  {10pt}
  {\it}
  {0pt}
  {\bf}
  {}
  {3mm}
  {}
\newtheoremstyle{definition}
  {10pt}
  {10pt}
  {}
  {0pt}
  {\bf}
  {}
  {3mm}
  {}
\theoremstyle{plain}
\definecolor{MyDarkBlue}{rgb}{0,0.29,0.7}
\hypersetup{
    colorlinks=true,       
    linkcolor=MyDarkBlue,          
    linkbordercolor=MyDarkBlue,
    citecolor=MyDarkBlue,
    citebordercolor=MyDarkBlue,        
    filecolor=MyDarkBlue,      
    urlcolor=MyDarkBlue           
}

\theoremstyle{plain}

\newtheorem{theorem}{Theorem}
\newtheorem{coro}[theorem]{Corollary}

\newtheorem{prop}[theorem]{Proposition}

\theoremstyle{definition}

\newtheorem{definition}[theorem]{Definition}

\newtheorem{remark}[theorem]{Remark}

\begin{document}

\title{Poncelet Curves}
\author[1]{Norbert Hungerb\"uhler}
\author[2]{Micha Wasem}
\affil[1]{Department of Mathematics, ETH Z\"urich, R\"amistrasse 101, 8092 Z\"urich, Switzerland,
norbert.hungerbuehler@math.ethz.ch\bigskip}
\affil[2]{School of Engineering and Architecture, HES-SO University of Applied Sciences and Arts Western Switzerland, P\'erolles 80, 1700 Freiburg, Switzerland,
micha.wasem@hefr.ch\vspace*{-9mm}}

\date{}
\maketitle
\begin{abstract}
\noindent We examine pairs of closed plane curves that have the same closing property as two conic sections in Poncelet's porism. 
We show how the vertex curve can be computed for a given envelope and vice versa. 
Our formulas are universal in the sense that they produce all possible sufficiently regular pairs of such Poncelet curves. 
We arrive at similar results for sets of curves, analogous to the pencil of conic sections in the full Poncelet  theorem.
We also study the case of Poncelet curves that carry  Poncelet polygons which are equiangular or even congruent.
\end{abstract}


\section{Introduction}
A popular version of  Poncelet's porism reads as follows:
\begin{theorem}\label{thm-poncelet}
Let $K$ and $C$ be nondegenerate conics in general position
which neither meet nor intersect. Suppose there is an $n$-sided polygon inscribed in $K$ and circumscribed about $C$.
Then for any point $P$ of $K$, there exists an $n$-sided polygon, also 
inscribed in $K$ and circumscribed about $C$, which has $P$ as one of its
vertices.
\end{theorem}
See, e.g.,~\cite{hhponcelet} for an elementary proof of the theorem and~\cite{vladimir,flatto} as general references. In this article we investigate 
pairs of curves which have the same closing property as the conics $K$ and $C$ in Theorem~\ref{thm-poncelet}:
\begin{definition}\label{def-pairs}
Let $K$ and $C$ be closed curves in the Euclidean plane $\mathbb R^2$. 
If every point $Q\in K$ is a vertex of an $n$-gon $P$ inscribed in $K$ and circumscribed about
$C$, then $(K,C)$ is called a {\em Poncelet pair}, and $P$ a {\em Poncelet polygon}. 
The curve $K$ is called {\em vertex curve}, the curve $C$ is called {\em envelope}.
Here, circumscribed also allows that the prolongation of the sides of $P$ are tangential to $C$ (see the example in Figure~\ref{figure-no-touch}).
\end{definition}
There is also the following, lesser known, full form of Poncelet's theorem (see~\cite[Theorem~5.2]{vladimir} and  Figure~\ref{fig-full-dual} on the left):
\begin{theorem}\label{thm-full}
Let the conics $K, C_1,\ldots , C_n$ belong to a pencil. If a polygon $P$
inscribed in $K$ exists such that each of its sides is tangent to one of the conics $C_1,\ldots , C_n$, 
then infinitely many such polygons exist which touch the conics $C_1,\ldots , C_n$ 
in the same order as $P$.
\end{theorem}
The dual form of this theorem is usually not mentioned in the literature. The reason
being that the dual of the pencil generated by two conics is {\em not\/} the pencil generated by the duals of
the two conics (see~\cite{hh}). Therefore the dual of Theorem~\ref{thm-full} is slightly less elegant compared to its original (see Figure~\ref{fig-full-dual} on the right):
\begin{theorem}\label{thm-dual}
Let $C, K_1,\ldots , K_n$ be conics such that their duals belong to a pencil.
If a polygon $P$  circumscribed about $C$ exists such that each of its vertices belongs to one of the conics $K_1,\ldots , K_n$, 
then infinitely many such polygons exist which have vertices on the conics $K_1,\ldots , K_n$ 
in the same order as $P$.
\end{theorem}
\begin{figure}[H]
\begin{center}
\raisebox{3mm}{\begin{tikzpicture}[line cap=round,line join=round,x=33,y=33]
\draw [red,rotate around={0.:(0.,0.)},line width=.8pt] (0.,0.) ellipse (3.1622776601683795 and 2.23606797749979);
\draw [rotate around={-31.71747441146101:(-0.6,0.2)},line width=.8pt] (-0.6,0.2) ellipse (1.6139964287134976 and 0.9975046506658882);
\draw [rotate around={-26.5650511770:(-0.44444444444,0.11111111111)},line width=.8pt] (-0.44444444444,0.11111111111) ellipse (2.0548046676563256 and 1.369869778437);
\draw [rotate around={-35.2862195687:(-0.774839319441,0.321405340279)},line width=.8pt] (-0.774839319441,0.321405340279) ellipse (0.911408391306 and 0.51552028438);

\draw [line width=.5pt,fill=black,fill opacity=.05] (-0.5756248322185297,2.198710537170979)-- (2.031697396051639,-1.7135060097468278)-- (-2.8187102352174196,-1.0136253276927722)--cycle;

\draw [line width=.5pt,dotted] (-2.5249797111144634,1.3461941647586206)-- (3.162070592610102,-0.025588741180447527);
\draw [line width=.5pt,dotted] (3.162070592610102,-0.025588741180447527)-- (-2.0045850975031207,-1.72939853516597);
\draw [line width=.5pt,dotted] (-2.0045850975031207,-1.72939853516597)-- (-2.5249797111144634,1.3461941647586206);
\begin{footnotesize}
\draw[color=red] (-2.1,1.8636885376243941) node {$K$};
\draw[color=black] (-1,-.5) node {$C_2$};
\draw [fill=white] (-0.5756248322185297,2.198710537170979) circle (1.5pt);
\draw[color=black] (0.6,1.3) node {$C_1$};
\draw [fill=white] (2.031697396051639,-1.7135060097468278) circle (1.5pt);
\draw [fill=white] (-2.8187102352174196,-1.0136253276927722) circle (1.5pt);
\draw[color=black] (-0.3108931379350341,0.4) node {$C_3$};
\draw [fill=white] (0.6386579063635754,0.3767121009940546) circle (1.5pt);
\draw [fill=white] (-0.17407053426699182,-1.3952287677327473) circle (1.5pt);
\draw [fill=white] (-1.5198610395775876,0.8463821634093095) circle (1.5pt);
\draw[color=black] (1.7,-.9) node {$a_2$};
\draw[color=black] (1.1,-1.7) node {$a_1$};
\draw[color=black] (-2.7,-.55) node {$a_3$};
\draw [fill=white] (3.162070592610102,-0.025588741180447527) circle (1.5pt);
\draw [fill=white] (-2.0045850975031207,-1.72939853516597) circle (1.5pt);
\draw [fill=white] (-2.5249797111144634,1.3461941647586206) circle (1.5pt);
\draw [fill=white] (0.3768763357745465,-0.9440631930448131) circle (1.5pt);
\draw [fill=white] (-2.3649462673177126,0.40037789646673083) circle (1.5pt);
\draw [fill=white] (-0.9598038517614331,0.9686197005106243) circle (1.5pt);
\end{footnotesize}
\end{tikzpicture}}
\begin{tikzpicture}[line cap=round,line join=round,x=33,y=33]
\draw [rotate around={-31.71747441146101:(-0.6,0.2)},line width=0.8pt] (-0.6,0.2) ellipse (3.01951082986131 and 1.8661603222526901);
\draw [rotate around={-29.976134849589837:(-0.3046875,0.1015625)},line width=0.8pt] (-0.3046875,0.1015625) ellipse (2.2275748568895235 and 1.4403657247435875);
\draw [rotate around={-25.950895629401764:(-0.1477116362548609,0.04923721208495364)},line width=0.8pt] (-0.1477116362548609,0.04923721208495364) ellipse (1.6850104948905924 and 1.1485527890317309);

\draw [line width=.5pt,fill=black,fill opacity=.05] (-0.14725821581491882,1.2615118845874542)-- (2.1487327374677267,-0.7660680351716186)-- (-1.67034020030658,-0.6622018739703512)--(-1.67034020030658,-0.6622018739703512)-- cycle;
\draw [red,rotate around={0.:(0.,0.)},line width=0.8pt] (0.,0.) ellipse (1. and 0.7071067811865476);

\draw [line width=0.5pt,dotted] (-1.6786876630898804,0.7518708071998076)-- (1.6229802017606323,0.6650927909572744);
\draw [line width=0.5pt,dotted] (1.6229802017606323,0.6650927909572744)-- (-0.004840861896711068,-1.5281486061060283);
\draw [line width=0.5pt,dotted] (-0.004840861896711068,-1.5281486061060283)-- (-1.6786876630898804,0.7518708071998076);
\begin{footnotesize}
\draw[color=red] (-0.45,0.46879961352047683) node {$C$};
\draw [fill=white] (-0.03843351743468511,-0.7065843420065285) circle (1.5pt);
\draw[color=black] (-1.7853484180042154,2.26) node {$K_1$};
\draw [fill=white] (2.1487327374677267,-0.7660680351716186) circle (1.5pt);
\draw[color=black] (2.35,-0.8) node {$A_1$};
\draw[color=black] (-1.3842691136370924,1.56) node {$K_2$};
\draw [fill=white] (-1.67034020030658,-0.6622018739703512) circle (1.5pt);
\draw[color=black] (-1.8311860527890294,-0.7745462300175698) node {$A_2$};
\draw[color=black] (-0.8685957223079346,1.1) node {$K_3$};
\draw [fill=white] (-0.14725821581491882,1.2615118845874542) circle (1.5pt);
\draw[color=black] (-0.2268688353205378,1.45) node {$A_3$};
\draw [fill=white] (-0.8725636534898562,0.3454219670261536) circle (1.5pt);
\draw [fill=white] (0.780556475657425,0.44200202958202556) circle (1.5pt);
\draw [fill=white] (1.6229802017606323,0.6650927909572744) circle (1.5pt);
\draw [fill=white] (-0.004840861896711068,-1.5281486061060283) circle (1.5pt);
\draw [fill=white] (-1.6786876630898804,0.7518708071998076) circle (1.5pt);
\draw [fill=white] (-0.8875387082634054,-0.32578753915253256) circle (1.5pt);
\draw [fill=white] (0.8854657558399458,-0.328595796712917) circle (1.5pt);
\draw [fill=white] (0.03700663479420996,0.7066224271070117) circle (1.5pt);
\end{footnotesize}
\end{tikzpicture}
\caption{Theorem~\ref{thm-full} for triangles with sides $a_1,a_2,a_3$ tangent to $C_1,C_2,C_3$ on the left, 
and Theorem~\ref{thm-dual} for triangles with vertices $A_1,A_2,A_3$ on $K_1,K_2,K_3$ on the right.}\label{fig-full-dual}
\end{center}
\end{figure}
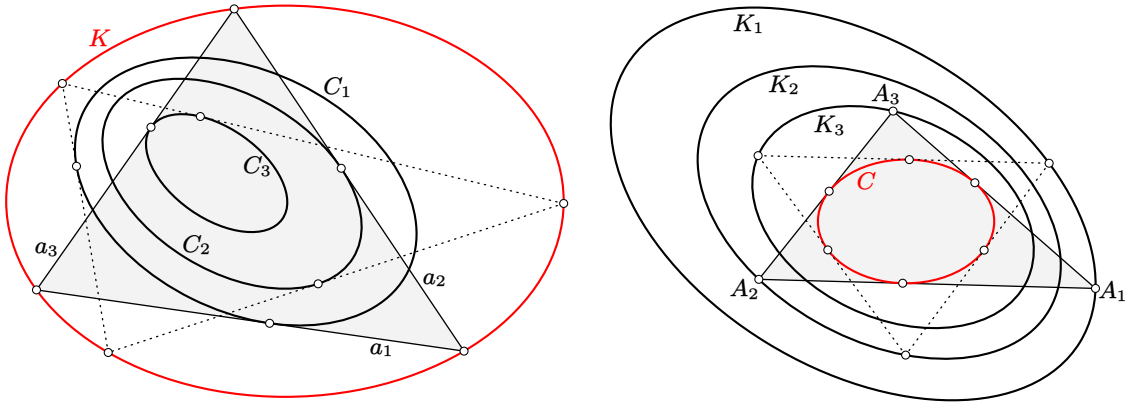
\begin{definition}\label{def-clan}
Generalizing Definition~\ref{def-pairs} we will call curves $(K, C_1,\ldots , C_n)$
which have the closing property described in Theorem~\ref{thm-full} or curves $(C, K_1,\ldots , K_n)$ 
which have the closing property in Theorem~\ref{thm-dual}  a {\em Poncelet clan}. 
\end{definition}

For the case when the vertex curve $K$ is a circle, 
Poncelet pairs $(K,C)$ were first  considered in~\cite{mirman98} and~\cite{gau} in 
the context of numerical ranges of matrices, and later intensively studied in 
\cite{mirman2001,mirman2009,chien,gorkin,gorkin2020},  \cite[\S5 and \S7]{wu-gau} and the survey paper~\cite{wu-gau-survey}. 
Another approach to Poncelet pairs $(K,C)$ where $K$ is again a circle and $C$ an algebraic curve uses Szeg\H{o} polynomials and is described in~\cite{hunziker}.
For a circle $K$, the corresponding envelope $C$ may also be constructed using Blaschke products (see~\cite{blaschke}).
While these methods seem to be limited to the special case where the vertex curve $K$ is a circle, 
we will take a more direct geometric point of view here
and consider methods to generate general Poncelet pairs and clans. 

This article is organized as follows. In Section~\ref{sec-2}, we consider Poncelet pairs and clans with
particularly nice geometric properties. In Section~\ref{sec-3} we will construct Poncelet pairs and clans
when either the vertex curve is prescribed, or when the envelope is prescribed. The formulas we develop are 
universal in the sense that any sufficiently regular Poncelet pair can be generated using the approach we present.

\section{Equiangular Poncelet polygons}\label{sec-2}

In Section~\ref{sec-ep} we will consider Poncelet pairs $(K,C)$ with  Poncelet polygons
whose external angles are all identical. Theorem~\ref{thm-6} describes the general solution for this situation. 
A special case occurs when the Poncelet polygons of $(K,C)$ are not only equiangular, but even equilateral and all congruent. 
This situation is covered in Theorem~\ref{thm-equilateral}. In Section~\ref{sec-ec} we will consider
equiangular Poncelet polygons for Poncelet clans $(C,K_1,K_2,\ldots,K_n)$.

\subsection{Equiangular pairs}\label{sec-ep}
If the external angles of each Poncelet polygon are equal and independent of the
starting point, then we say that the Poncelet pair $(K,C)$ is {\em equiangular}. 
\begin{theorem}\label{thm-6}
If $C$ is a closed $C^2$ curve in the Euclidean plane with non-vanishing curvature \textcolor{black}{and total curvature $2k\pi$, where $0< k\in\N$}, and $0<\alpha<\pi$ is commensurable with $\pi$, 
then there exist $2k$ \textcolor{black}{closed curves $K_i$, $i=0,\ldots,2k-1$ such that $(K_i,C)$ is an equiangular Poncelet pair with
angle $\alpha$.}
\end{theorem}
\begin{proof}
We can assume that $C$ is parametrized by 
\begin{equation}\label{eq-C}
C:\  [0,2k\pi)\to \mathbb R^2,\quad \varphi\mapsto X(\varphi)=p(\varphi)u(\varphi)+p'(\varphi)u'(\varphi)
\end{equation}
for a $C^2$ support function $p:[0,2k\pi)\to \mathbb R$ as indicated in Figure~\ref{fig-support}. 
Here $2k\pi$ is the total curvature of $C$, $u(\varphi)=(\cos(\varphi),\sin(\varphi))$, and $\rho(\varphi) = p(\varphi) + p''(\varphi)>0$ is the radius of curvature of $C$ in 
the point $X(\varphi)$ (see, e.g., \cite{ahw}).

\begin{figure}[h!]
\begin{center}
\definecolor{blue}{rgb}{0.,0.,1.}
\definecolor{red}{rgb}{1.,0.,0.}
\definecolor{ccwwqq}{rgb}{0.8,0.4,0.}
\begin{tikzpicture}[line cap=round,line join=round,x=0.9cm,y=0.9cm]
\clip(-4.16,-3.01) rectangle (5.74,3.);
\draw [->,shorten >=1pt,shift={(1.95914692164,-1.64394227432)}] (0,0) -- (0.:0.9) arc (0.:55.26533603178533:0.9);
\draw [shift={(3.825687642976,1.04820339101)}] (0,0) -- (-124.73466396821468:0.35) arc (-124.73466396821468:-34.73466396821468:0.35);
\draw[line width=.5pt,color=ccwwqq,smooth,samples=100,domain=0.0:6.283185307179586] plot ({3.8*cos(deg(\x))-0.3*sin(3.0*deg(\x))},{2.8*sin(deg(\x))+0.2*cos(2.0*deg(\x))});
\draw [domain=-4.16:5.74] plot(\x,{(-7.786461130023362--1.4588134273615847*\x)/-2.1040731660110015});
\draw [line width=.8pt,domain=1.9591469216489055:5.740000000000002] plot(\x,{(-5.360654225861506-0.*\x)/3.2608530783510945});
\fill(3.86,0.84) circle (0.03);
\draw [line width=1pt,color=red,xshift=-.1,yshift=.1] (3.8256876429761633,1.0482033910159714)-- (1.9591469216489055,-1.6439422743241814);
\draw [color=red,decoration={
        brace,
        mirror,
        raise=0.15cm,pre=moveto, pre length=2pt, post=moveto,post length=2pt
    },
    decorate] (3.8256876429761633,1.0482033910159714)-- (1.9591469216489055,-1.6439422743241814) node [pos=0.5,anchor=south,xshift=-13.5pt,yshift=0.pt] {\begin{footnotesize}$p(\varphi)$\end{footnotesize}};

\draw [line width=1pt,color=blue] (3.8256876429761633,1.0482033910159714)-- (2.3485560108715426,2.072340523556868);
    
    \draw [color=blue,decoration={
        brace,
        mirror,
        raise=0.15cm
    },
    decorate] (3.8256876429761633,1.0482033910159714)-- (2.3485560108715426,2.072340523556868) node [pos=0.5,anchor=south,xshift=15pt,yshift=3pt] {\begin{footnotesize}$p'(\varphi)$\end{footnotesize}};

\draw [-{>[scale=1.1]},color=darkgreen,line width=.8pt] (1.9591469216489055+.018,-1.6439422743241814-.018) -- (2.891009427511566+.018,-0.29990001626549634-.018);
\draw [-{>[scale=1.1]},color=orange,line width=.8pt] (1.9591469216489055,-1.6439422743241814) -- (.615,-.712);
\begin{footnotesize}
\draw [fill=black] (2.3485560108715426,2.072340523556868) circle (1.0pt);
\draw[color=black] (2.4,2.15) node[anchor=north east] {$X$};
\draw[color=black] (4.,-1.89) node[anchor=west] {horizontal};
\draw[color=black] (4.,-2.3) node[anchor=west] {axis};

\draw [fill=black,xshift=-.16] (3.8256876429761633,1.0482033910159714) circle (1pt);
\draw [fill=black] (1.9591469216489055,-1.6439422743241814) circle (1.0pt);
\draw[color=black] (1.7,-1.79) node {$S$};
\draw[color=black] (2.5,-1.4) node {$\varphi$};
\draw[color=black] (3.4,-0.51) node {$\textcolor{darkgreen}{u(\varphi)}$};
\draw (.515,-1.1) node {$\textcolor{orange}{u'(\varphi)}$};
\draw (-3,-2) node[xshift=5pt] {$\textcolor{ccwwqq}{C}$};
\end{footnotesize}
\end{tikzpicture}
\caption{The support function $p$  of the curve $C$.}\label{fig-support}
\end{center}
\end{figure}
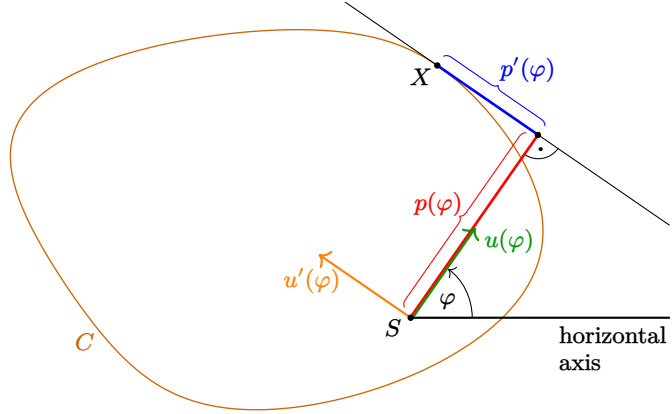

The condition that the circumscribed Poncelet polygon has external angle $\alpha$
translates into the equation
\begin{equation}\label{eq-euler}
p(\varphi) u(\varphi) + 
 q(\varphi) u'(\varphi) =p(\varphi+\alpha) u(\varphi+\alpha) + 
   s(\varphi) u'(\varphi+\alpha)
\end{equation}
for some function $s$ (see Figure~\ref{fig-2}). By multiplying~(\ref{eq-euler}) by $u(\varphi+\alpha)$,
we get 
\begin{equation}\label{eq-q}
q(\varphi)=\csc(\alpha)\bigl( p(\varphi + \alpha)  -\cos(\alpha) p(\varphi)\bigr),
\end{equation}
and hence
\begin{equation}\label{eq-K}
K:\  [0,2k\pi)\to \mathbb R^2,\quad \varphi\mapsto Y(\varphi)=
\csc(\alpha) \bigl(p(\varphi + \alpha) u'(\varphi) - p(\varphi) u'(\varphi + \alpha)\bigr)
\end{equation}
is an explicit $C^2$ parametrization of a possible vertex curve $K$.
\textcolor{black}{Note that \eqref{eq-K} also yields a vertex curve such that the Poncelet polygon has external angle $\alpha$, if it is applied to the angle $\alpha_i = \alpha + i\pi$, $i=0,\ldots,2k-1$. In this way, we obtain $2k$ Poncelet pairs $(K_i,C)$, where
\begin{equation}\label{eq-Ki}
K_i:\  [0,2k\pi)\to \mathbb R^2,\quad \varphi\mapsto Y_i(\varphi)=
\csc(\alpha_i) \bigl(p(\varphi + \alpha_i) u'(\varphi) - p(\varphi) u'(\varphi + \alpha_i)\bigr).
\end{equation}
See Figure~\ref{fig-2k} for an example with $k=1, \alpha=2\pi/3$.
}

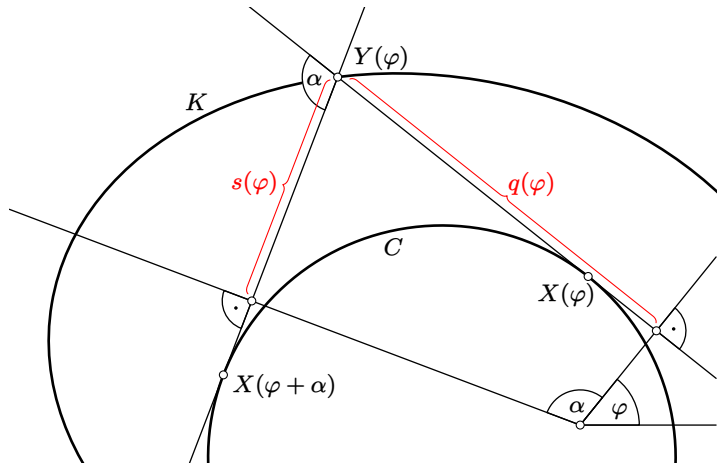
\begin{figure}[h!]
\begin{center}
\begin{tikzpicture}[line cap=round,line join=round,>=triangle 45,x=33,y=33]
\clip(-2.422822468778543,-0.8205349613616574) rectangle (5.595834348030821,4.382205581986205);

\draw [shift={(4.053784960182866,-0.3914429577865759)},line width=0.5pt] (0,0) -- (51.26367489787726:0.4022737533516404) arc (51.26367489787726:159.14957510924467:0.4022737533516404) -- cycle;
\draw [shift={(0.32799516197702117,1.0276052189443543)},line width=0.5pt] (0,0) -- (159.14957510924467:0.3352281277930336) arc (159.14957510924467:249.14957510924467:0.3352281277930336) -- cycle;
\draw [shift={(4.918258577071045,0.6861956384067843)},line width=0.5pt] (0,0) -- (-38.73632510212273:0.3352281277930336) arc (-38.73632510212273:51.263674897877266:0.3352281277930336) -- cycle;
\draw [shift={(4.053784960182866,-0.3914429577865759)},line width=0.5pt] (0,0) -- (0.:0.6704562555860673) arc (0.:51.26367489787726:0.6704562555860673) -- cycle;
\draw [shift={(2.4845810773250525,-0.7616369905418275)},line width=1pt,samples=100]  plot[domain=-0.34555990765957656:3.4412815293759813,variable=\t]({1.*2.652201292588279*cos(\t r)+0.*2.652201292588279*sin(\t r)},{0.*2.652201292588279*cos(\t r)+1.*2.652201292588279*sin(\t r)});
\draw [line width=.5pt,domain=4.053784960182866:5.595834348030821] plot(\x,{(-1.9728725072443418-0.*\x)/5.04});
\draw [line width=.5pt,domain=4.053784960182866:5.595834348030821] plot(\x,{(-5.998386654836674--1.3733206625998138*\x)/1.1016675576938164});
\draw [line width=.5pt,domain=-2.422822468778543:4.053784960182866] plot(\x,{(-4.294081977906017--1.4190481767309302*\x)/-3.7257897982058448});
\draw [rotate around={-9.424718464396593:(2.3669903638573646,0.27862269718654603)},line width=1.pt] (2.3669903638573646,0.27862269718654603) ellipse (4.367319374590568 and 3.3171931669370878);

\draw[fill=white,white] (1.12,3.5) circle(7pt);
\draw [line width=0.5pt,domain=-2.422822468778543:5.595834348030821] plot(\x,{(--9.581856271364549-1.6595813230220262*\x)/2.0688067886304418});

\draw [line width=.5pt,domain=-2.422822468778543:5.595834348030821] plot(\x,{(-0.21579774777555258-3.404251335311012*\x)/-1.2965832513775193});

\fill[line width=0.5pt] (0.14806037428079716,0.9469295852408006) circle (0.02);
\fill[line width=0.5pt] (5.11427441548313,0.707710209942433) circle (0.02);
\draw [shift={(1.302648294511109,3.5866111480809137)},line width=.5pt] (0,0) -- (141.26367489787725:0.4022737533516404) arc (141.26367489787725:249.14957510924467:0.4022737533516404) -- cycle;
 \draw [color=red,decoration={
        brace,
        ,
        raise=0.08cm,pre=moveto, pre length=2pt, post=moveto,post length=2pt
    },
    decorate] (0.3279951619770211,1.0276052189443543)-- (1.302648294511109,3.5866111480809137) node [pos=0.5,anchor=south,xshift=-15pt,yshift=-6pt] {\begin{footnotesize}$s(\varphi)$\end{footnotesize}};
    
    \draw [color=red,decoration={
        brace,
        mirror,
        raise=0.08cm,pre=moveto, pre length=2pt, post=moveto,post length=4pt
    },
    decorate] (4.918258577071045,0.6861956384067843)-- (1.302648294511109,3.5866111480809137) node [pos=0.5,anchor=south,xshift=13pt,yshift=-1pt] {\begin{footnotesize}$q(\varphi)$\end{footnotesize}};

\begin{footnotesize}
\draw[color=black] (1.935143192530894,1.65) node {$C$};
\draw [fill=white] (4.144162400347079,1.3071697980886143) circle (1.5pt);
\draw[color=black] (3.9,1.1) node {$X(\varphi)$};
\draw [fill=white] (1.302648294511109,3.5866111480809137) circle (1.5pt);
\draw[color=black] (1.8,3.81) node {$Y(\varphi)$};
\draw [fill=white] (4.053784960182866,-0.3914429577865759) circle (1.5pt);
\draw [fill=white] (0.3279951619770211,1.0276052189443543) circle (1.5pt);
\draw[color=black] (1.04,3.55) node {$\alpha$};
\draw[color=black] (4.02,-0.18) node {$\alpha$};
\draw[color=black] (-0.3,3.3) node {$K$};
\draw [fill=white] (0.006065043133589878,0.1823598127699017) circle (1.5pt);
\draw[color=black] (0.7,0.057762733456087456) node {$X(\varphi+\alpha)$};
\draw [fill=white] (4.918258577071045,0.6861956384067843) circle (1.5pt);
\draw[color=black] (4.5,-0.23) node {$\varphi$};
\end{footnotesize}
\end{tikzpicture}
\caption{Construction of equiangular Poncelet pairs.}\label{fig-2}
\end{center}
\end{figure}
\end{proof}

\begin{remark}
The Poncelet polygons in the Poncelet pair $(K_i,C)$, $i=0,\ldots,2k-1$, in Theorem~\ref{thm-6}
are given as follows: The Poncelet polygon with starting point $Y_i(\varphi)\in K_i$ has the vertices
$Y_i(\varphi+j\alpha_i)$, where $j=0,1,\ldots, m-1$. Here $m$ is minimal such that
$m\alpha_i$ is a  multiple of $2k\pi$. This implies that this Poncelet polygon has
\begin{equation}\label{counting_vertices_1}
m = \frac{2\pi}{\alpha_i}\mathrm{lcm}\left(\frac{\alpha_i}{2\pi},k\right)
\end{equation}
vertices. We adopt the notation $\operatorname{lcm}$ for the least common multiple. 

\end{remark}

\begin{figure}[h!]
\begin{center}
\begin{tikzpicture}[line cap=round,line join=round,x=5,y=5]
\draw[line width=0.8pt] (4.603745677565161,-17.576720695796684) -- (13.428108737407618,11.526967054417641) -- (-16.188619246620327,4.617244834086516) -- cycle;

\draw [line width=1pt,domain=0:6.33,samples=100,blue] plot({1.1547*(cos((0.523599 +\x) r)*(9. + 0.9*cos((2*\x) r) - 0.222222*cos((5*\x) r) + 0.285714*sin((3*\x) r)) - sin(\x r)*((9+0.9*cos(2*(8.37758 + \x) r)) +0.222222*sin((0.523599-5*\x) r)+0.285714*sin((3*\x) r)))}, 
{ 1.1547*(cos(\x r)*(9+0.9*cos(2*(8.37758 + \x) r)+0.22222*sin((0.523599 - 5*\x) r)+0.285714*sin((3*\x) r))+(9+0.9*cos((2*\x) r)-0.22222*cos((5*\x) r)+0.285714*sin((3*\x) r))*sin((0.523599 + \x) r))});


\draw [line width=0.5pt,fill=black,fill opacity=.05] (4.6037456775,-17.5767206957)-- (13.4281087374,11.5269670544) -- (-16.1886192466,4.6172448340)-- cycle;
\draw [line width=1pt,domain=0:6.33,samples=100,red] plot({cos(((\x r)))*(9+9/10*cos((2*(\x r)))-2/9*cos((5*(\x r)))+2/7*sin((3*(\x r))))-sin(((\x r)))*(6/7*cos((3*(\x r)))-9/5*sin((2*(\x r)))+10/9*sin((5*(\x r))))},{sin(((\x r)))*(9+9/10*cos((2*(\x r)))-2/9*cos((5*(\x r)))+2/7*sin((3*(\x r))))+cos(((\x r)))*(6/7*cos((3*(\x r)))-9/5*sin((2*(\x r)))+10/9*sin((5*(\x r))))});

\draw [line width=0.5pt,dotted] (-3.300594259181336,-18.22162764613673)-- (16.098139674089175,5.073106904164468);
\draw [line width=0.5pt,dotted] (16.098139674089175,5.073106904164468)-- (-13.775057354431185,10.2255353455177);
\draw [line width=0.5pt,dotted] (-13.775057354431185,10.2255353455177)-- (-3.300594259181336,-18.22162764613673);

\begin{footnotesize}
\draw[color=red,xshift=-6pt] (7.755804516064713,4.4114739592064405) node {$C$};
\draw[color=blue,xshift=6pt] (10.170599900053288,16.074421877619315) node {$K$};

\end{footnotesize}
\end{tikzpicture}
\qquad\quad
\begin{tikzpicture}[line cap=round,line join=round,x=7,y=7]
\clip(-14,-13) rectangle (14,12);

\draw [line width=1pt,domain=0:6.33,samples=100,red] plot({cos(((\x r)))*(9+9/10*cos((2*(\x r)))-2/9*cos((5*(\x r)))+2/7*sin((3*(\x r))))-sin(((\x r)))*(6/7*cos((3*(\x r)))-9/5*sin((2*(\x r)))+10/9*sin((5*(\x r))))},{sin(((\x r)))*(9+9/10*cos((2*(\x r)))-2/9*cos((5*(\x r)))+2/7*sin((3*(\x r))))+cos(((\x r)))*(6/7*cos((3*(\x r)))-9/5*sin((2*(\x r)))+10/9*sin((5*(\x r))))});

\draw [line width=0.4pt,dash pattern=on 3pt off 3pt,] (9.167702954792123,5.015396368707265)-- (-2.3549918821543483,9.495206662154017);
\draw [line width=0.4pt,dash pattern=on 3pt off 3pt,] (-2.3549918821543483,9.495206662154017)-- (-10.46825063480926,2.982508187808282);
\draw [line width=0.4pt,dash pattern=on 3pt off 3pt,] (-10.46825063480926,2.982508187808282)-- (-9.12150636038046,-5.762505668196141);
\draw [line width=0.4pt,dash pattern=on 3pt off 3pt,] (-9.12150636038046,-5.762505668196141)-- (1.5847888919338917,-9.924918879273427);
\draw [line width=0.4pt,dash pattern=on 3pt off 3pt,] (1.5847888919338917,-9.924918879273427)-- (10.381128183194487,-2.8638859589095285);
\draw [line width=0.4pt,dash pattern=on 3pt off 3pt,] (10.381128183194487,-2.8638859589095285)-- (9.167702954792123,5.015396368707265);

\draw[line width=.8pt,fill=black,fill opacity=0.05] (1.5847888919,-9.92491887) -- (10.38112818,-2.86388595) -- (12.098000805390411,-14.012264783655416) -- cycle;
\draw[line width=.8pt,fill=black,fill opacity=0.05] (9.167702954,5.01539636) -- (-2.354991882,9.495206662) -- (7.285977607146112,17.23423351622083) -- cycle;
\draw[line width=.8pt,fill=black,fill opacity=0.05] (-10.468250634,2.98250818) -- (-9.12150636,-5.76250566) -- (-17.368285446809566,-2.5563076799100988) -- cycle;


\draw [dash pattern=on 3pt off 3pt,line width=0.4pt,domain=-30.76614479250185:21.542308069724513] plot(\x,{(-35.469307919804955-3.543761256390596*\x)/0.5457464178498657});
\draw [dash pattern=on 3pt off 3pt,line width=0.4pt,domain=-30.76614479250185:21.542308069724513] plot(\x,{(-81.4288634994879--5.837687060396917*\x)/7.272322924722058});
\draw [dash pattern=on 3pt off 3pt,line width=.4pt,domain=-30.76614479250185:21.542308069724513] plot(\x,{(-136.9273834872706--6.2047914868549565*\x)/-15.959595912365497});

\draw [shift={(10.381128183194487,-2.8638859589095285)},line width=0.4pt] (0,0) -- (98.75488486543357:1.7) arc (98.75488486543357:218.75491373694894:1.7) -- cycle;
\draw [shift={(1.5847888919338917,-9.924918879273427)},line width=0.4pt] (0,0) -- (-141.24498209368255:1.7) arc (-141.24498209368255:-21.245070472813495:1.7) -- cycle;
\draw [shift={(-9.12150636038046,-5.762505668196141)},line width=0.4pt] (0,0) -- (-21.245266912042037:1.7) arc (-21.245266912042037:98.75485069012626:1.7) -- cycle;
\draw [shift={(-10.46825063480926,2.982508187808282)},line width=0.4pt] (0,0) -- (98.7549172635084:1.7) arc (98.7549172635084:218.7549370670535:1.7) -- cycle;
\draw [shift={(9.167702954792123,5.015396368707265)},line width=0.4pt] (0,0) -- (-21.24514821107853:1.7) arc (-21.24514821107853:98.75498103178045:1.7) -- cycle;
\draw [shift={(-2.3549918821543483,9.495206662154017)},line width=0.4pt] (0,0) -- (-141.24519007246343:1.7) arc (-141.24519007246343:-21.245159427270774:1.7) -- cycle;

\draw [line width=1pt,domain=0:6.33,samples=100,blue] plot({10.3923*cos((0.523599 + \x) r) + 1.03923*cos(2*\x r)*cos((0.523599 + \x) r) - 
 0.2566*cos(5*\x r)*cos((0.523599 + \x) r) + 10.3923*sin(\x r) + 
 1.03923*cos(2*(5.23599 + \x) r)*sin(\x r) - 
 0.2566*cos(5*(5.23599 + \x) r)*sin(\x r) + 
 0.329914*cos((0.523599 + \x) r)*sin(3*\x r) - 
 0.329914*sin(\x r)*sin(3*\x r)},{ -10.3923*cos(\x r) - 
 1.03923*cos(\x r)*cos(2*(5.23599 + \x ) r) + 
 0.2566*cos(\x r)*cos(5*((5.23599 + \x) ) r) + 0.329914*cos(\x r)*sin(3*\x r) + 
 10.3923*sin((0.523599 + \x) r) + 1.03923*cos(2*\x r)*sin((0.523599 + \x) r) - 
 0.2566*cos(5*\x r)*sin((0.523599 + \x) r) + 
 0.329914*sin(3*\x r)*sin((0.523599 + \x) r)});
 
\begin{footnotesize}
\draw[color=red,xshift=-3pt] (7.755804516064713,4.4114739592064405) node {$C$};
\draw[color=blue,xshift=2pt] (11.785522528458975,0.5730303179372186) node {$K$};
\end{footnotesize}
\end{tikzpicture}

\caption{Illustration for Theorem~\ref{thm-6}. Solutions for $k=1$: On the left an equiangular Poncelet triangle, shown in two positions.
On the right an equiangular Poncelet hexagon. Both solutions have external angle $\alpha=2\pi/3$. The hexagon can also
be seen as a hexagon with external angle $\alpha=\pi/3$.}\label{fig-2k}
\end{center}
\end{figure}
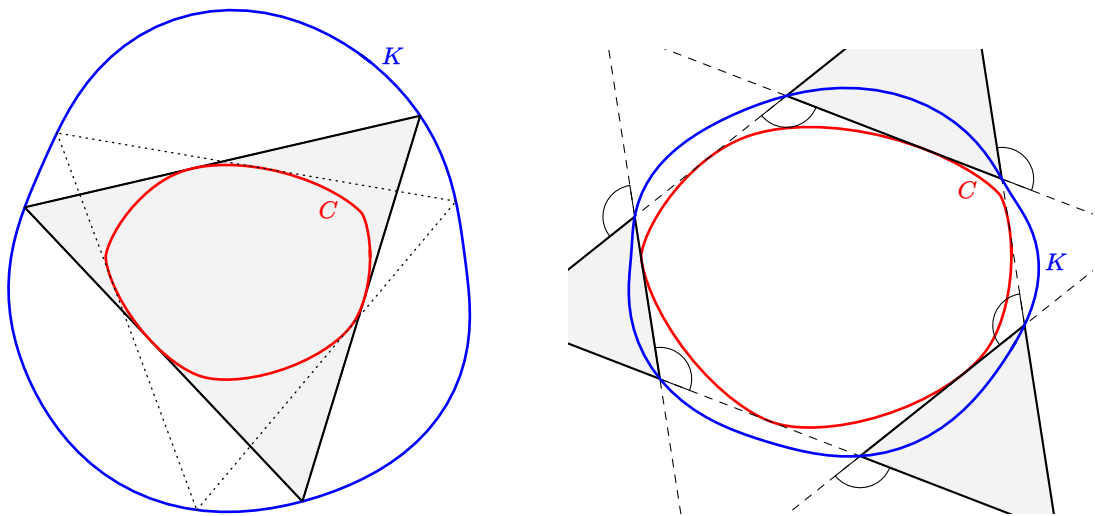
The vertex curves constructed in Theorem~\ref{thm-6} may exhibit singularities.
%
However, the following theorem shows that this only occurs in very special cases.
\begin{theorem}
If the given envelope $C$ has no self-intersections, then the vertex curve $K_i$ in  Theorem~\ref{thm-6}
is regular. If the given envelope $C$ has self-intersections and
all intersection angles are different from the angle $\alpha_i$, the vertex curve is regular. A singularity of $K$ can only
occur in a self-intersection of $C$.
\end{theorem}
\begin{proof}
If we take the derivative in~(\ref{eq-Ki}), we find
$$
Y'_i(\varphi) = \csc(\alpha_i)\bigl(
p'(\varphi+\alpha_i) u'(\varphi) - p(\varphi+\alpha_i) u(\varphi)
-p'(\varphi) u'(\varphi+\alpha_i) - p(\varphi) u(\varphi+\alpha_i)
\bigr).
$$
We fix the parameter $\varphi$.
By shifting the origin $S$ to a point on the curve $C$ and rotating the coordinate system,
we may assume $\varphi=0$ and $p(0)=p'(0)=0$. Hence, a singularity $Y'_i(0)=0$ can only occur, if
$p(\alpha_i)=p'(\alpha_i)=0$. 
\end{proof}
An example of a singular vertex curve is shown in Figure~\ref{fig-singular}.
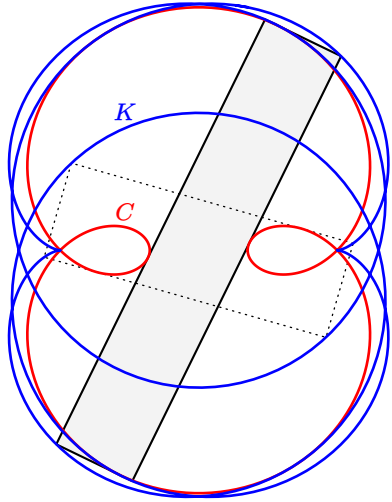
\begin{figure}[H]
\begin{center}
\begin{tikzpicture}[line cap=round,line join=round,x=55,y=55]
\draw[line width=0.8pt,fill=black,fill opacity=.05] (0.9694909067487621,1.3309583837226813) -- (-0.45192708283524596,-1.583391110399402) -- (-0.9694909067487847,-1.3309583837226697) -- (0.4519270828352749,1.5833911103993876) -- cycle;

\draw[line width=0.5pt,dotted] (1.0515849254160368,0.04759082265017423) -- (-0.866576420982722,0.5976143191075957) -- (-1.051584954541141,-0.047590814298689) -- (0.8665763252445536,-0.5976142916551467) -- cycle;

\draw [line width=1pt,domain=0:18.9,samples=400,red] plot({1/6*(5*cos(\x/3 r)-4*cos(\x r) + cos(5*\x/3 r)},{1/6*(5*sin(\x/3 r)-4*sin(\x r)+sin(5*\x/3 r))});

\draw [line width=1pt,domain=0:18.9,samples=400,blue] %
plot({(-2/3+cos(2*\x/3 r))*cos(\x r)-1/3*(2+3*cos(2*\x/3 r))*sin(\x r)},{-1/3*(cos(\x/3 r)+sin(\x/3 r))*(cos(\x/3 r)+sin(\x/3 r)) *(cos(\x/3 r)+sin(\x/3 r))*(-5+6*sin(2*\x/3 r))});
\begin{footnotesize}
\draw[color=red] (-.5,.26) node {$C$};
\draw[color=blue] (-.5,.95) node {$K$};
\end{footnotesize}
\end{tikzpicture}
\caption{An example of a vertex curve $K$ with singularities (cusps) for a Poncelet pair $(C,K)$ 
which carries rectangles, i.e., Poncelet polygons with exterior angle $\alpha=\pi/2$. The reason for the 
singularities is that the envelope $C$ has self-intersections
with intersection  angle of the same size $\pi/2$.
Note that here the extensions of the sides touch the envelope.}\label{fig-singular}
\end{center}
\end{figure}

Observe that the proof of Theorem~\ref{thm-6} yields explicit 
formulas~(\ref{eq-C}) and~(\ref{eq-K}) for an equiangular Poncelet pair $(K,C)$.
A simple, but particularly interesting family of examples is obtained by choosing
$p(\varphi) = a+b\cos(\ell \varphi)$. Multiplying $p$ by a constant results in 
a homothety of the curves $C$ and $K$, hence we may assume $b=1$ in the sequel. 
The condition $p(\varphi) + p''(\varphi)>0$ is satisfied if we choose
$a>\ell^2-1$ if $\ell>1$, and $a>1-\ell^2$ if $0<\ell<1$. 
In this situation, we get by~(\ref{eq-K}) 
for $\ell=2,\alpha=\frac{2\pi}3$, the Poncelet pair 
\begin{eqnarray*}
C:[0,2\pi) \to\mathbb R^2,&&\varphi\mapsto u(\varphi)(a + \cos(2 \varphi)) - 2 u'(\varphi) \sin(2 \varphi)\\
K:[0,2\pi) \to\mathbb R^2,&&\varphi\mapsto 2a u(\varphi)-u(3\varphi).
\end{eqnarray*}
Hence, the vertex curve $K$ is an epitrochoid. It turns out that the Poncelet triangle in this case is not only equiangular,
but even equilateral. This is a special case of Theorem~\ref{thm-equilateral} below.
Figure~\ref{fig-3}
shows on the left an example with $a=\frac85$.
Observe that for $a=2+\sqrt3$ we get the classical
design of a Wankel engine (see Figure~\ref{fig-3} on the right): Here, the constant $a$ is chosen as the smallest
value such that a Reuleaux triangle can turn inside the vertex curve $K$. In general we have the following.
\begin{theorem}\label{thm-equilateral}
Let $0<k\in\mathbb N$, $0<\ell\in\mathbb Q$, \textcolor{black}{$\ell\ne \frac{2k}{m}-1, m\in\Z$}, $n=\ell+1,\alpha=\frac{2k\pi}{n}$ 
and $p(\varphi)=a+\cos(\ell \varphi)$ with $a>\ell^2-1$ if $\ell>1$, and $a>1-\ell^2$ if $0<\ell<1$.
Then the Poncelet polygon $P$ for the Poncelet pair $(K,C)$ given by~(\ref{eq-C}) and~(\ref{eq-K})
is equiangular  with angle $\alpha$ and equilateral with side length 
$2a|\tan(\alpha/2)|$.
$K$ is an epitrochoid. 
The midpoint of $P$ rotates on a circle centered at the origin with radius $1$.
\end{theorem}
\begin{proof}
We first use $p(\varphi)=a+\cos(\ell \varphi)$ in the general formula~(\ref{eq-K}).
Unfortunately this does not yield a very telling expression. However, if we
replace $\varphi$ by $\varphi-\alpha/2$ and observe that $k$ is an integer, we get the reparametrization
\begin{equation}\label{eq-KK}
K:\  [0,2k\pi)\to \mathbb R^2,\quad \varphi\mapsto Y(\varphi)=
a \sec(\alpha/2) u(\varphi) + (-1)^k u(n\varphi ).
\end{equation}
This shows that the vertex curve $K$ is an epitrochoid. The vertices of the Poncelet polygon $P$ are the points
$$\bigl\{Y(\varphi+j\alpha)=(-1)^ku(n\varphi)+a\sec(\alpha/2)u(\varphi+j\alpha), 
j=1,\ldots,\tfrac n{(n, k)}\bigr\}$$
where $(n,k)$ denotes the greatest common divisor of $n$ and $k$.
Its midpoint, given by $(-1)^ku(n\varphi)$, rotates on the unit circle.
The perimeter radius of $P$ is $a | \sec(\alpha/2)|$ and the center angle over a side
is $\alpha$. Hence we get for the side  length $s$ of $P$
$$
s=2a|\sec(\alpha/2) \sin(\alpha/2)|=2a\left|\tan\left(\alpha/2\right)\right|.
$$
\end{proof}
Figure~\ref{fig-5} illustrates Theorem~\ref{thm-equilateral} with two examples.
The first one is a Wankel engine with three combustion chambers.
\textcolor{black}{We note that in \cite[p.~515ff]{glaeser} a variant of a Wankel engine design with a spherical triangular rotor turning in an equilateral spherical conic is discussed.}
\begin{figure}[h!]
\begin{center}
\begin{tikzpicture}[line cap=round,line join=round,x=34.2,y=34.2]

\draw [line width=1pt,domain=0:6.3,samples=100,red] plot({((1.6+cos(2*\x r)/2)*cos(\x r)+sin(2*\x r)*sin(\x r)},{(1.6+cos(2*\x r)/2)*sin(\x r)-sin(2*\x r)*cos(\x r))});


\draw [line width=1pt,domain=0:6.28319,samples=100,blue] plot(
{1/2 *cos(3*\x r) + 8/5*(cos(\x r)- sqrt(3)*sin(\x r))},
{8/5*sqrt(3)*cos(\x r) + (8*sin(\x r))/5 + 1/2*sin(3*\x r)}
);
 

\draw[line width=.8pt,fill=black,fill opacity=.05]  (0.489, 3.658)-- (-2.7688, -0.826)--(2.7434, -1.4054)--cycle;
\draw[line width=.5pt,dotted] (-1.6555, 3.1704)-- (-2.4269, -2.3183)-- (2.7121, -0.242)--cycle;

\begin{footnotesize}

\draw[color=red] (1,.7) node {$C$};
\draw[color=blue] (2,3.2) node {$K$};

\end{footnotesize}
\end{tikzpicture}\qquad
\begin{tikzpicture}[line cap=round,line join=round,scale=.525]

\draw [line width=1pt,domain=0:6.28319,samples=100,blue] plot(
{cos(\x r)*(5 + 2*sqrt(3) - 2*cos(2*\x r))},
{(3 + 2*sqrt(3) - 2*cos(2*\x r))*sin(\x r)}
);
 
 
  \tkzDefPoint(6.514, 1.6868){A}
  \tkzDefPoint(-6.2178, 3.9318){B}
  \tkzDefPoint(-1.7961, -8.2167){C}

\tkzDrawPolygon[fill=black!5,color=black!5](A,B,C)
 \tkzDrawArc[fill=black!5,thick](A,B)(C)
 \tkzDrawArc[fill=black!5,thick](B,C)(A)
 \tkzDrawArc[fill=black!5,thick](C,A)(B)

\def\eps{.07}
\draw [line width=.5pt,domain=0:6.28319,samples=600,black,fill=white] plot(
{-.5+cos(\x r)*(3+sin(45*(\x-\eps) r)/15)},
{-.866+sin(\x r)*(3+sin(45*(\x-\eps) r)/15)}
);

\draw [line width=.5pt,domain=0:6.28319,samples=600,black,fill=black,fill opacity=.1] plot(
{cos(\x r)*(2+sin(30*\x r)/15)},
{sin(\x r)*(2+sin(30*\x r)/15)}
);


\begin{footnotesize}
\draw[color=blue] (4.5,7.5) node {$K$};


\end{footnotesize}
\end{tikzpicture}

\caption{A Poncelet pair with $k=1,n=3, p(\varphi)=\frac85+\frac12\cos(2\varphi)$ on the left,
the design of a Wankel engine on the right.
}\label{fig-3}
\end{center}
\end{figure}
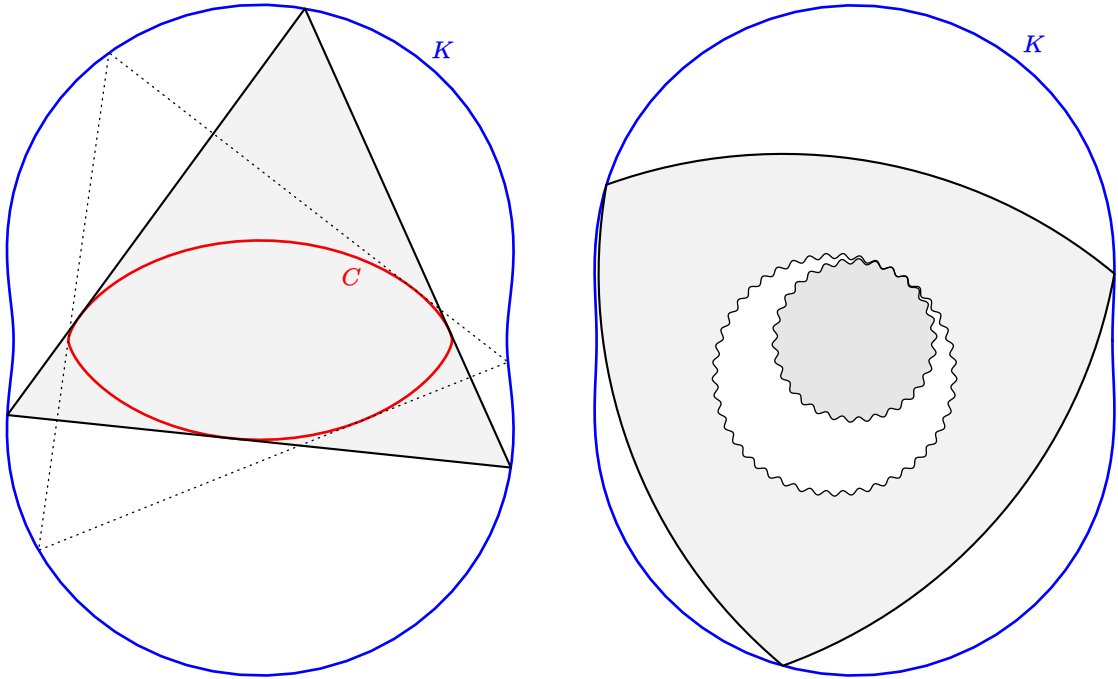
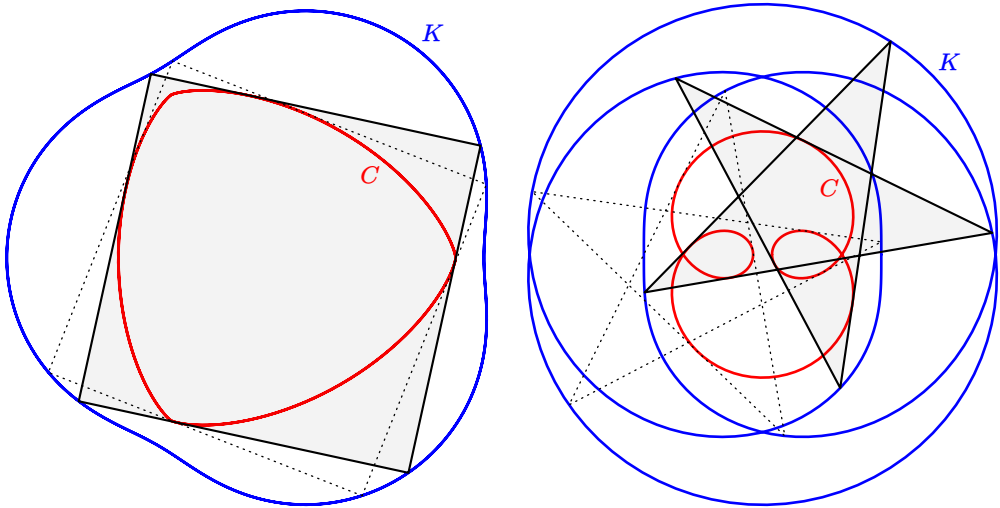
\begin{figure}[h!]
\begin{center}
\begin{tikzpicture}[line cap=round,line join=round,x=7.8,y=7.8]
\draw [line width=1pt,domain=0:18.8496,samples=200,red] plot(
{cos(\x r)*(8.1 + cos(3*\x r)) + 3*sin(\x r)*sin(3*\x r)},
{-8*(-1.0125 + (cos(\x r))^3)*sin(\x r}
);

\draw [line width=1pt,domain=0:18.8496,samples=300,blue] plot(
{11.4551*cos(\x r) - cos(4*\x r)}, {11.4551*sin(\x r) - sin(4*\x r)}
);

\tikzset{declare function={f1(\j)=-cos(4*\x r) + 11.4551*cos((1.5708*\j+ \x) r);}};
\tikzset{declare function={f2(\j)=-sin(4*\x r) + 11.4551*sin((1.5708*\j+ \x) r);}};
\def\x{.57}
\draw [line width=.8pt,fill=black,fill opacity=.05]({f1(1)},{f2(1)}) -- ({f1(2)},{f2(2)})--({f1(3)},{f2(3)})--({f1(4)},{f2(4)})--cycle;
\def\x{3.55}
\draw[line width=.5pt,dotted] ({f1(1)},{f2(1)}) -- ({f1(2)},{f2(2)})--({f1(3)},{f2(3)})--({f1(4)},{f2(4)})--cycle;
\begin{footnotesize}
\draw[color=blue] (8,10.9) node {$K$};
\draw[color=red] (5,4) node {$C$};
\end{footnotesize}
\end{tikzpicture}
\quad
\begin{tikzpicture}[line cap=round,line join=round,x=25,y=25]
\draw [line width=1pt,domain=0:18.8496,samples=200,red] plot(
{(-25/36*(-1+sqrt(5))+cos(2*\x/3 r))*cos(\x r)+2/3*sin(2*\x/3 r)*sin(\x r)},
{1/36*(30*sin(\x/3 r)-25*(-1+sqrt(5))*sin(\x r)+6*sin(5*\x/3 r))}
);

\draw [line width=1pt,domain=0:18.8496,samples=300,blue] plot(
{2.77778*cos(\x r)- cos(1.66667*\x r)}, {2.77778*sin(\x r) - sin(1.66667*\x r)}
);

\tikzset{declare function={f1(\j)=-cos(1.66667*\x r) + 2.77778*cos((3.76991*\j+ \x) r);}};
\tikzset{declare function={f2(\j)=-sin(1.66667*\x r) + 2.77778*sin((3.76991*\j+ \x) r);}};
\def\x{2.37}
\draw [line width=.8pt,fill=black,fill opacity=.05]({f1(1)},{f2(1)}) -- ({f1(2)},{f2(2)})--({f1(3)},{f2(3)})--({f1(4)},{f2(4)})--({f1(5)},{f2(5)})--cycle;
\def\x{3.94}
\draw[line width=.5pt,dotted] ({f1(1)},{f2(1)}) -- ({f1(2)},{f2(2)})--({f1(3)},{f2(3)})--({f1(4)},{f2(4)})--({f1(5)},{f2(5)})--cycle;

\begin{footnotesize}
\draw[color=blue] (2.8,2.9) node {$K$};
\draw[color=red] (1,1) node {$C$};
\end{footnotesize}
\end{tikzpicture}

\caption{A Wankel engine with three combustion chambers with parameters $k=1,\ell=3, n=4, a>\ell^2-1$ on the left, 
a regular pentagram as Poncelet polygon resulting from the parameters $k = 4, \ell = 2/3, n = l + 1, a = \cos(\alpha/2) n^2$ on the right.}\label{fig-5}
\end{center}
\end{figure}

\begin{remark}
In Figure~\ref{fig-5} on the right the parameter $a$ has been chosen minimal
with the property that the curvature of the curve $K$ does not change sign.
Indeed, if 
$a=\cos(\alpha/2)n^2$, then the curvature of $K$ is non-negative and there holds
$$
\det\begin{pmatrix}Y'(\varphi)& Y''(\varphi)\end{pmatrix}=2n^3(n+1)\begin{cases}\cos^2\left(\frac{\ell}{2}\varphi\right),&\text{if $k$ is odd}\\
\hfill \sin^2\left(\frac{\ell}{2}\varphi\right), &\text{if $k$ is even.}\end{cases}
$$
And if $a>\cos(\alpha/2)n^2$, then the curvature of $K$ is strictly positive.
\end{remark}
\subsection{Equiangular clans}\label{sec-ec}
Using the results from the previous section it is now straightforward to construct equiangular Poncelet clans.
Note that we generalize the definition of the term {\em equiangular\/} for clans in the following sense:
We no longer require that all exterior angles of the Poncelet polygon are equal for all starting points, 
but only that the angle in each vertex does not depend on the starting point.

\begin{coro}\label{coro-11}
Let $C$ be a closed $C^2$ curve in the Euclidean plane with non-vanishing curvature, given by~(\ref{eq-C}).
Let $\alpha_1,\alpha_2,\ldots, \alpha_n$ be given external angles such that $\alpha_1+\ldots+\alpha_n = 2m\pi$,
for a natural number $m\ge 1$. Then the curves $(C,K_1,K_2,\ldots,K_n)$ form a Poncelet clan, where
the vertex curve $K_i$ is given by the parametrization
$$
K_i:\  [0,2k\pi)\to \mathbb R^2,\quad  \varphi\mapsto Y_i(\varphi) = \csc(\alpha_i)\bigl(p(\varphi + \alpha_i)u'(\varphi) -p(\varphi)u'(\varphi+\alpha_i)\bigr).
$$
The polygon $P(\varphi)$ with vertices $Y_1(\varphi),Y_2(\varphi+\alpha_1),\ldots,Y_n(\varphi+\alpha_1+\ldots+\alpha_{n-1})\color{black},Y_1(\varphi+2\pi m),Y_2(\varphi+2\pi m + \alpha_1),\ldots,Y_n(\varphi+2\pi m+\alpha_1+\ldots+\alpha_{n-1}), Y_1(\varphi+4\pi m),\ldots$ 
is a Poncelet polygon for each value of $\varphi$.
\end{coro}
Figure~\ref{fig-can} shows an example of such a Poncelet clan with three vertex curves $K_1,K_2,K_3$ and $k=2$, $m=1$.
\begin{figure}[H]
\begin{center}
\begin{tikzpicture}[line cap=round,line join=round,x=23,y=23]

\draw[line width=.5pt,dotted] (2.44423, 0.0979303)-- (-0.625738, 0.759942)-- (-1.94781, -2.03675)-- 
(5.634, 0.206043)-- (-0.50411, 1.52967)-- (-1.77029, -1.14878)-- cycle;

\draw[line width=.8pt,fill=black,fill opacity=.07] (-2.85395, 2.03352)--(-0.254474, -1.63231)-- (3.47336, -0.110361)-- (-3.3082,   3.19887)-- (-0.458934, -0.819213)--(1.42006, -0.0520878)-- (-2.85395, 
  2.03352)-- (-0.254474, -1.63231)-- (3.47336, -0.110361)-- (-3.3082,   3.19887)-- (-0.458934, -0.819213)-- (1.42006, -0.0520878)-- cycle;

\draw [line width=1pt,domain=0:12.57,samples=200,red] plot(
{0.375*cos(.5*\x r) + cos(\x r) + 0.125*cos(1.5*\x r)},
{0.375*sin(.5*\x r) +  sin(\x r) + 0.125*sin((1.5*\x r)});

\draw [line width=1pt,domain=0:12.57,samples=200,blue] plot(
{-2.08583*sin(\x r) -  1.04291*cos((0.5*(2.64159 + \x)) r)*sin(\x r) +  2.08583*sin((2.64159 + \x) r) + 1.04291*cos(0.5*\x r)*sin((2.64159 + \x) r)}, 
{2.08583*cos(\x r) + 1.04291*cos((0.5*(2.64159 + \x)) r)*cos(\x r)  -  2.08583*cos((2.64159 + \x) r) - 1.04291*cos(0.5*\x r)*cos((2.64159 + \x) r)}
);

\draw [line width=1pt,domain=0:12.57,samples=200,darkgreen] plot(
{-1.02685*sin(\x r) - 0.513427*cos((0.5*(1.34159 + \x)) r)*sin(\x r) +  1.02685*sin((1.34159 + \x) r) + 0.513427*cos(0.5*\x r)*sin((1.34159 + \x) r)}, 
{1.02685*cos(\x r) + 0.513427*cos((0.5*(1.34159 + \x)) r)*cos(\x r) - 1.02685*cos((1.34159 + \x) r) - 0.513427*cos(0.5*\x r)*cos((1.34159 + \x) r)}
);
   
\draw [line width=1pt,domain=0:12.57,samples=200,brown] plot(
{-1.34101*sin(\x r) - 0.670506*cos((0.5*(2.3 + \x)) r)*sin(\x r)  + 1.34101*sin((2.3 + \x) r) + 0.670506*cos(0.5*\x r)*sin((2.3 + \x) r)}, 
{1.34101*cos(\x r) + 0.670506*cos((0.5*(2.3 + \x)) r)*cos(\x r) - 1.34101*cos((2.3 + \x) r) - 0.670506*cos(0.5*\x r)*cos((2.3 + \x) r)}
);

\begin{footnotesize}
\draw[color=blue] (4.8,3.5) node {$K_1$};
\draw[color=darkgreen] (1.6,1.5) node {$K_2$};
\draw[color=brown] (2.9,2.4) node {$K_3$};

\draw[color=red] (0,1.13) node {$C$};
\end{footnotesize}
\end{tikzpicture}
\caption{A Poncelet clan according to Corollary~\ref{coro-11}.}\label{fig-can}
\end{center}
\end{figure}
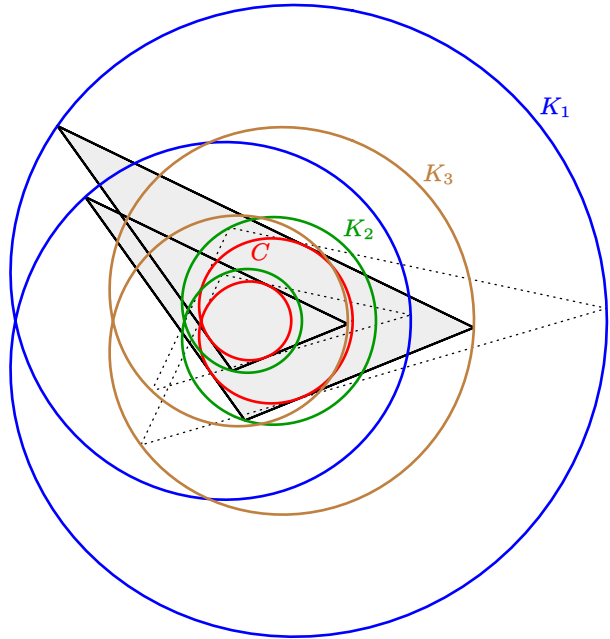

\begin{proof}
The formula for the vertex curves $K_i$ follows immediately from~\eqref{eq-K}.
The polygon $P(\varphi)$ is by construction tangential to the envelope $C$ and its vertices lie on $K_1,\ldots,K_n$. 
\end{proof}
Note that if the Poncelet polygon $P$ is a triangle, then the triangles $P(\varphi)$ are similar for all $\varphi$.
\textcolor{black}{
\begin{remark}
Like in Theorem \ref{thm-6}, each angle $\alpha_i$ in Corollary \ref{coro-11} can be replaced by
$$
\alpha_{j_i}=\alpha_i+ j_i\pi,
$$
where $j_i=0,\ldots,2k-1$.
Upon writing
$$
K_i^{j_i}:\  [0,2k\pi)\to \mathbb R^2,\quad  \varphi\mapsto Y_i^{j_i}(\varphi) = \csc(\alpha_{j_i})\bigl(p(\varphi + \alpha_{j_i})u'(\varphi) -p(\varphi)u'(\varphi+\alpha_{j_i})\bigr)
$$
we obtain all possible Poncelet clans with external angles $\alpha_1,\ldots,\alpha_n$ as
$$
(C,K_1^{j_1},K_2^{j_2},\ldots, K_n^{j_n}),
$$
where $j_i\in\{0,\ldots,2k-1\}$ leading to $(2k)^n$ Poncelet clans. For a fixed multi-index $\mathbf j = (j_1,\ldots,j_n)\in \N^n$ with entries in $\{0,\ldots,2k-1\}$, we can write down the vertices of the corresponding Poncelet polygon $P_{\mathbf j}(\varphi)$ as follows: Let $l$ be the smallest natural number such that
$$
l\sum_{i=1}^n \alpha_{j_i} = l\sum_{i=1}^n \left(\alpha_i + j_i\pi\right) = l \pi(2m+|\mathbf j|)
$$
is a multiple of $2\pi m$, hence
$$
l = \frac{\operatorname{lcm}(2m+|\mathbf j|,2k)}{2m+|\mathbf j|}.
$$
The vertices of $P_{\mathbf j}(\varphi)$ are are obtained by going through the rows of the $(l\times n)$-matrix $(\mathbf V_{\ell \nu})_{\mycom{1\leqslant \ell \leqslant l}{1\leqslant \nu\leqslant n}}$ with entries
$$
\mathbf{V}_{\ell \nu} = Y^{j_\nu}_\nu\left(\varphi + \pi(2m+|\mathbf j|)(\ell-1)+\sum_{i=1}^{\nu-1}\alpha_{j_i}\right),
$$
hence the number of vertices of the polygon is at most
\begin{equation}\label{counting_vertices_2}
\frac{n\cdot\operatorname{lcm}(2m+|\mathbf j|,2k)}{2m+|\mathbf j|}.
\end{equation}
Note that the formula \eqref{counting_vertices_2} yields the correct number of vertices if all angles $\alpha_1,\ldots,\alpha_n$ are pairwise distinct, \textcolor{black}{which is generically the case}, however, there are degenerate cases, where there are less vertices: For instance, if $\alpha_1=\alpha_2=\ldots=\alpha_n = \frac{2\pi m}{n}$ and $\mathbf j = (i,i,\ldots, i)$ for some $i\in\{0,\ldots,2k-1\}$, then all the curves of the corresponding Poncelet clan are identical and the situation reduces to the one in Theorem \ref{thm-6}, where the formula \eqref{counting_vertices_1} yields the correct number of vertices
$$
\frac{n}{2m+in}\operatorname{lcm}\left(\frac{2m+2ni}{n},2k\right),
$$
whereas fomula \eqref{counting_vertices_2} yields
$$
\frac{n}{2m+in}\operatorname{lcm}(2m+2ni,2k).
$$
\end{remark}}

\section{Construction of general Poncelet pairs and clans}\label{sec-3}
In this section we will construct general Poncelet pairs and clans. 
In Section~\ref{sec-3.2} and~\ref{sec-3.3} we specify the vertex curve $K$ and find corresponding envelopes, 
while in Section~\ref{sec-3.4} and~\ref{sec-3.5} we specify the envelope and find corresponding vertex curves.
The necessary tools are provided in Section~\ref{sec-3.1}.

\subsection{Torsion maps of \boldmath$S^1$}\label{sec-3.1}
In order to use standard terminology and to avoid the unnecessary appearance of the factor $2\pi$ throughout, we will use the identification $S^1\cong \R/\mathbb Z$ in this section.

The following construction of Poncelet pairs will make use of diffeomorphisms $f:S^1\to S^1$ with the property that the $n$-th iterate of $f$ equals the identity on $S^1$, i.e.\ $f^0=f^n=\mathrm{id}_{S^1}$, but $f^i\ne\mathrm{id}_{S^1}$ for $0<i<n$. We will call such a map $f$ a {\em torsion map} and we will without loss of generality restrict ourselves to orientation preserving diffeomorphisms.

As an example of a torsion map, consider the rotation $r_\alpha :S^1\to S^1$ about the angle $\alpha=\frac{m}{n}\in \mathbb Q$, where $\frac mn$ is a reduced fraction. Conjugating $r_\alpha$ by any diffeomorphism $h:S^1\to S^1$ yields another torsion map: If $f=h^{-1} \circ r_\alpha \circ h$, then $f^n = h^{-1}\circ r_\alpha^n\circ h = \mathrm{id}_{S^1}$.

We will now show that all 
torsion maps $f$ arise from rotations up to conjugation by diffeomorphisms of $S^1$.

Every orientation preserving diffeomorphism $f:S^1\to S^1$ lifts to an increasing diffeomorphism $F:\R\to \R$ such that $\pi\circ F = f\circ \pi$, where $\pi: \R \to S^1$, $x\mapsto x\mod 1$ is the usual quotient map and $F(x+1) = F(x) + 1$ for all $x\in\R$. Poincar\'e introduced the rotation number of $f$ as
$$
\tau(f)=\lim_{k\to\infty}\frac{F^k(x)-x}{k}\mod 1
$$
and showed that this limit always exists and that it is independent of the lift and $x\in\R$ (see~\cite{poincare}).

To proceed we now need the following result. A variant of it appears  for instance in \cite[Ex. 14 (b)]{Sanhueza2018}.
For the readers convenience we provide an explicit proof.

\begin{prop}\label{rotationnumberprop}Let $f:S^1\to S^1$ be a torsion map, i.e., 
an orientation preserving diffeomorphism such that $f^n = \mathrm{id}_{S^1}$ and $f^i\ne \mathrm{id}_{S^1}$ for $0<i<n$. Then we have:
\begin{enumerate}
\item
The rotation number of $f$ equals $\frac{m}{n}$ for some $m\in\mathbb Z$ and the fraction $\frac{m}{n}$ is reduced.
\item $f$ is conjugate to the rotation $r_{\frac mn}$ about the angle $\frac mn$, i.e.\ there is an orientation preserving diffeomorphism $h:S^1\to S^1$ such that $f = h^{-1} \circ r_{\frac mn}\circ h$.
\end{enumerate}
\end{prop}
\begin{proof}
For the first item, consider the unique lift $F$ of $f$ such that $F(0)\in[0,1)$: If $0\le x < 1$ we find inductively
$
F^i(x) < i+1
$
%
%
and since $f^n = \mathrm{id}_{S^1}$ we have $F^n(x) = x+m$ for some $m\in \mathbb Z$ and $m\leq n$. If $k=n d$ for some $d\in\mathbb N$, it follows that
$$
\frac{F^k(x)-x}{k} = \frac{x+md-x}{nd}=\frac mn
$$
and therefore $\tau(f) = \frac mn$ (or, if $m=n=1$, $\tau(f)=0$ in which case $f=\mathrm{id}_{S^1}$).

If the rotation number of an orientation preserving diffeomorphism is $\frac{m}{n}\in \Q$, where $\frac mn$ is reduced, then every periodic orbit of $f$ has period $n$. Therefore, $\frac mn$ is reduced as the minimal length of a period of $f$ is $n$ by assumption \cite[p.12]{turer}.

For the second item we will construct a diffeomorphism $h:S^1\to S^1$ which conjugates $f$ to the rotation $r_{\frac mn}$ explicitly: The rotation $r_{\frac mn}:S^1\to S^1$ lifts to the increasing diffeomorphism $R_{\frac mn}:\R\to\R, x\mapsto x+\frac mn$. Define the increasing diffeomorphism $H:\R\to\R$,
$$
H = \frac{1}{n}\sum_{j=0}^{n-1} F^{j}.
$$
By construction, if $F\in C^k$, then $H\in C^k$ as well and $H\circ  F = R_{\frac mn}\circ H$. Since $H(x+1) = H(x) +1$, the map $H$ descends to an orientation preserving diffeomorphism $h: S^1\to S^1$.

Since $F = H^{-1}\circ R_{\frac mn}\circ H$ we obtain
$$
\pi \circ F = \pi \circ  H^{-1}\circ R_{\frac mn}\circ H = h^{-1}\circ r_{\frac mn}\circ h \circ \pi$$ so that $F$ is a lift of $h^{-1}\circ r_{\frac mn}\circ h$ and we conclude that
$
f = h^{-1}\circ r_{\frac mn}\circ h.
$
\end{proof}

\textcolor{black}{Before we proceed to construction of Poncelet pairs using Proposition~\ref{rotationnumberprop}, we remark that the connection between conjugacy to a rotation about a rational angle and Poncelet's theorem already appears in~\cite[Lemma 1.3]{king}, where Poncelet's theorem is recast in a measure theoretic framework for the case where $C$ and $K$ are nested ellipses. In a different context, exploiting dynamical billiards of an ellipse, the connection between a rational rotation number and periodicity of a billiard trajectories is established in~\cite{kolodziej}. Also in the context of billiard dynamics, for a given circle $K$, \cite{lopes} contructs a family of nested circles $C$ such that $(K,C)$ is a Poncelet pair. This article also provides a formula that counts the number of such Poncelet pairs, where the Poncelet polygon closes after $n$ steps.}

\subsection{Poncelet pairs for given vertex curve \boldmath$K$}\label{sec-3.2}

In order to construct Poncelet pairs, we will change back to the convention $S^1\cong \R/(2\pi \mathbb Z)$ in all that follows.

Let $K:S^1\to \R^2$, $\varphi\mapsto Y(\varphi)$, be a positively oriented regular $C^k$ curve, $k\ge 2$,
with non-vanishing curvature and let $f\in C^k(S^1,S^1)$ be a torsion map such that $f^0=f^n = \mathrm {id}_{S^1}$, where $n\geq 2$. 
According to the discussion in the previous section, there exists a $C^k$ diffeomorphism $h :S^1\to S^1$ such that
$$
f = h^{-1}\circ r_{\frac{2\pi m}{n}}\circ h,
$$
where $n> m\in\mathbb N$. For each fixed $\varphi\in S^1$, the function $f$ gives rise to a polygon $P$ with vertices
$$
Y(\varphi),Y(f(\varphi)),\ldots,Y(f^{n-1}(\varphi))
$$
and its edges are the segments between $Y(f^i(\varphi))$ and $Y(f^{i+1}(\varphi))$ where $i$ is taken mod $n$. 
For fixed $\varphi$ we consider the line
$$
s\mapsto E(s,\varphi)=(1-s)Y(\varphi)+sY(f(\varphi))
$$
connecting two consecutive edges of the polygon $P$.
The envelope $C:S^1\to\R^2, \varphi\mapsto X(\varphi)$, of these lines
when $\varphi$ varies is obtained by requiring that $X(\varphi) = E(s(\varphi),\varphi)$ is tangential to $\partial_sE(s,\varphi)$ for all $\varphi\in S^1$. This condition translates formally into
\begin{equation}\label{definition-s}
s= - \frac{\langle Y',J(Y\circ f-Y)\rangle}{\langle (Y\circ f-Y)',J(Y\circ f-Y)\rangle},
\end{equation}
where $
J = \left(\begin{smallmatrix}0 & -1 \\ 1& 0\end{smallmatrix}\right)
$.
Hence, the envelope is a regular $C^{k-1}$ curve $C:S^1\to\R^2$, given by
\begin{equation}\label{envelopeC}
X = Y-\frac{\langle Y',J(Y\circ f-Y)\rangle}{\langle (Y\circ f-Y)',J(Y\circ f-Y)\rangle} (Y\circ f-Y),
\end{equation}
whenever $s$ has no singularities. 
If $0\leq s\leq 1$, then the contact point of the line $E(s,\varphi)$ with the envelope $C$ lies on the 
segment joining the two vertices $Y(\varphi)$ and $Y(f(\varphi))$, otherwise on its prolongation.

Since $f=h^{-1}\circ r_\alpha \circ h$ with $0<\alpha\notin 2\pi\N$ being a rational multiple of $2\pi$, we consider $Z=Y\circ h^{-1}$. With the notation $t=h(\varphi)$ and $\Delta(t) = Z(t+\alpha)-Z(t)$, the formula for $X$ as a function of $t$ becomes
\begin{equation}\label{envelopeniceeq}
X = Z-\frac{\langle Z',J \Delta\rangle}{\langle \Delta',J\Delta\rangle }\Delta.
\end{equation}
Here, $X$ is well-defined if and only if $\langle\Delta',J\Delta\rangle\ne 0$ and in this case, $X\in C^{k-1}(S^1,\R^2)$ since $Z$ is of class $C^k$.
\begin{prop}\label{regularityofX}\leavevmode
\begin{enumerate}
\item Suppose that $X$ is well-defined. Then it is a regular curve, i.e.\ $X'\ne 0$, if and only if
$$
0\ne \det\begin{pmatrix}
\langle Z',J\Delta\rangle & \langle Z'_\alpha,J\Delta\rangle \\
\langle Z'',J\Delta\rangle + 2\langle Z',J\Delta'\rangle & \langle Z_\alpha'',J\Delta\rangle + 2\langle Z_\alpha',J\Delta'\rangle
\end{pmatrix},
$$
where $Z_\alpha(t) = Z(t+\alpha)$.

\item If $Z$ is of class $C^3$, $X$ is regular if and only if the curvature of $X$ is never zero.
\end{enumerate}
\end{prop}
\begin{proof}
We have $\langle X',J\Delta\rangle = 0$. Hence $X$ is regular if and only if
$\langle X',J\Delta'\rangle\ne 0$ since $\Delta$ and $\Delta'$ are linearly independent. 
Multiplying this condition by $\langle \Delta',J\Delta\rangle\ne 0$ yields the equivalent condition
$$0\ne\langle Z'_\alpha,J\Delta \rangle\left(\langle Z'',J\Delta\rangle + 2\langle Z',J\Delta'\rangle\right)-\langle Z',J\Delta \rangle\left(\langle Z_\alpha'',J\Delta\rangle + 2\langle Z',J\Delta'\rangle\right),
$$
which can be restated as claimed, since $\langle Z'_\alpha,J\Delta'\rangle = \langle Z',J\Delta'\rangle$.

If $Z$ is of class $C^3$, $X$ is of class $C^2$. If $X'\ne 0$ it follows from $\langle X',J\Delta\rangle = 0$ that
$$
X' =\underbrace{ \frac{\langle X',\Delta \rangle}{\|\Delta\|^2}}_{=:g}\Delta
$$
where $g\in C^1(S^1)$ never vanishes. Hence
$$
\langle X'',JX'\rangle = \langle g'\Delta+g\Delta',gJ\Delta\rangle = g^2\langle \Delta',J\Delta\rangle\ne 0,
$$
since $X$ is well-defined. Conversely, if $\langle X'',JX'\rangle\ne 0$, $X'$ cannot vanish.
\end{proof}
The next theorem is the main result of this section. It shows how to construct a 
Poncelet pair $(K,C)$ if the vertex curve $K$ and a torsion map $f:S^1\to S^1$ are given.
\begin{theorem}\label{poncelet-pair-no-self-intersection}
Let $K:S^1\to\R^2$, $\varphi\mapsto Y(\varphi)$, be a positively oriented regular $C^k$ curve, $k\geq 2$, with non-vanishing curvature.
Then the pair $(K,C)$, where the envelope $C\in C^{k-1}(S^1,\R^2)$ is defined via~\eqref{envelopeC}, is a Poncelet pair.
If $k\geq 3$ and the curvature of $C$ does not vanish, then $C$ is a regular curve.
If $K$  has no self-intersections, and $k\geq 2$, then the function $s$ defined in~\eqref{definition-s} satisfies $0<s<1$,
i.e., the points of contact of the Poncelet polygon with the envelope $C$ lie inside the sides.
\end{theorem}
\begin{proof}
Only the last part of the claim remains to be proven. So, assume that $K$ has no self-intersections.
We parametrize $C$ via~\eqref{envelopeniceeq} and show that $\langle\Delta',J\Delta\rangle>0$, which implies that $C$ is well-defined. Since $Z$ has no self-intersections, it parametrizes the boundary of a convex set and hence it lies entirely on one side of any of its tangents. This together with the fact that $Z$ is positively oriented and has non-vanishing curvature translates into the fact that for any $0<\alpha<2\pi$, we have
\begin{equation}\label{convexity-estimates}
\langle \Delta(t),JZ'(t)\rangle > 0\text{ and }\langle-\Delta(t),JZ'(t+\alpha)\rangle>0.
\end{equation}
Adding the two inequalities gives
$
0<\langle \Delta(t),J(Z'(t)-Z'(t+\alpha))\rangle = \langle \Delta',J\Delta\rangle.
$

We are left to show that $s\in(0,1)$, where $s=-\frac{\langle Z',J \Delta\rangle}{\langle \Delta',J\Delta\rangle}$. According to the first inequality of~\eqref{convexity-estimates}, we have $-\langle Z',J\Delta\rangle = \langle JZ',\Delta\rangle>0$ so that $s>0$. The condition $s<1$ is equivalent to
$$
-\langle Z'(t),J\Delta(t)\rangle<\langle \Delta'(t),J\Delta(t)\rangle \Longleftrightarrow 0<  \langle Z'(t+\alpha),J\Delta(t)\rangle,
$$
which holds true by the second inequality of~\eqref{convexity-estimates}.
\end{proof}
The next result states that a given Poncelet pair $(K,C)$ induces a torsion map $f$.
\begin{prop}\label{pro-reg}
Let $(K,C)$ be a Poncelet pair for Poncelet $n$-gons. Assume
that $K$ is a $C^2$ curve parametrized by $\gamma:S^1\to \R^2$
with $\dot\gamma\neq 0$, and $C$ a $C^2$ curve with non-vanishing curvature.
Let  $P_1=\gamma(t_1), P_2=\gamma(t_2)$ be two consecutive vertices of a Poncelet polygon.
Then the map $f: S^1\to S^1, t_1\mapsto t_2$, constitutes an orientation preserving $C^2$ diffeomorphism with
$f^n=\operatorname{id}_{S^1}$ and $f^i\neq \operatorname{id}_{S^1}$  for $0<i<n$.
\end{prop}
\begin{proof}
By applying a projective map if necessary, we may assume the situation depicted in Figure~\ref{fig-diffeo}.
\begin{figure}[H]
\begin{center}
\begin{tikzpicture}[line cap=round,line join=round,x=33,y=33]
\draw[line width=1.pt,smooth,samples=100,domain=-3.5:4] plot(\x,{(-(\x)^(2))/8+2});
\draw [samples=150,line width=1.pt,domain=-2.4:3.4)] plot (\x,{1-(\x-.5)^2/6});
\draw [line width=0.5pt] (3.287459472959304,0.649076276706267)-- (3.287459472959304,-0.92);
\draw [line width=0.5pt] (0.9074619995093508,0.972329119825974)-- (0.9074619995093508,-0.92);
\draw [line width=0.5pt] (-2.200894140934368,1.3945081225500964)-- (-2.200894140934368,-0.92);
\draw [line width=0.5pt] (3.287459472959304,0.649076276706267)-- (-2.200894140934368,1.3945081225500964);
\draw [line width=0.5pt,->] (-3.5,-0.92)-- (4,-0.92);
\begin{footnotesize}
\draw[color=black] (-3.24,0.41) node {$K$};
\draw[color=black] (-1.4,.2) node {$C$};
\draw [fill=white] (3.287459472959304,0.649076276706267) circle (1.5pt) node[above,xshift=14pt] {$(x,g(x))$};
\draw [fill=white] (0.9074619995093508,0.972329119825974) circle (1.5pt) node[above] {$(\xi,h(\xi))$};
\draw [fill=white] (-2.200894140934368,1.3945081225500964) circle (1.5pt) node[above,xshift=-11pt] {$(\mu,g(\mu))$};
\draw [fill=white] (-2.200894140934368,-0.92) circle (1.5pt) node[below] {$\mu$};
\draw [fill=white] (3.287459472959304,-0.92) circle (1.5pt) node[below] {$x$};
\draw [fill=white] (0.9074619995093508,-0.92) circle (1.5pt) node[below] {$\xi$};
\end{footnotesize}
\end{tikzpicture}
\caption{Construction of the diffeomorphism $f: S^1\to S^1$.}\label{fig-diffeo}
\end{center}
\end{figure}
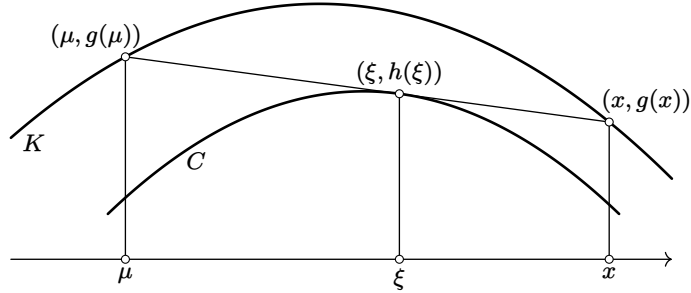

Here, the vertex curve $K$ is locally given as the graph of a $C^2$ function $g$, and the envelope $C$
as graph of a $C^2$ function $h$. We are interested in the regularity of the map $x\mapsto\mu$,
as defined in Figure~\ref{fig-diffeo}. We have
\begin{eqnarray*}
g(x)-h(\xi) &=& h'(\xi)(x-\xi)\\
h(\xi)-g(\mu) &=& h'(\xi)(\xi-\mu).
\end{eqnarray*}
Consider the $C^1$ map 
$$
F:\R^{1+2}\to\R^2, \quad(x,\xi,\mu)\mapsto\begin{pmatrix}
g(x)-h(\xi) - h'(\xi)(x-\xi)\\
h(\xi)-g(\mu)- h'(\xi)(\xi-\mu)
\end{pmatrix}
$$
The Jacobian matrix of $F$ is given by
$$
DF(x,\xi,\mu) = \begin{pmatrix}
g'(x)-h'(\xi) & (\xi-x)h''(\xi) & 0\\
0 & (\mu-\xi)h''(\xi) & h'(\xi) - g'(\mu).
\end{pmatrix}
$$
In particular, we have
$$
\det\begin{pmatrix}
 (\xi-x)h''(\xi) & 0\\
 (\mu-\xi)h''(\xi) & h'(\xi) - g'(\mu) 
\end{pmatrix}= (\xi-x)h''(\xi)(h'(\xi)-g'(\mu)) \neq 0
$$
and hence the implicit function theorem guarantees the local existence of 
a $C^1$ function $x\mapsto (\xi(x),\mu(x))$ such that $F(x,\xi(x),\mu(x))=0$, locally.
However, for the partial function $\mu(x)$ we get a better regularity, namely $C^2$, since
\begin{eqnarray*}
\begin{pmatrix}
\xi'(x)\\ \mu'(x)
\end{pmatrix} &=& 
-\begin{pmatrix}
(\xi(x)-x)h''(\xi(x)) & 0\\
 (\mu(x)-\xi(x))h''(\xi(x)) & h'(\xi(x)) - g'(\mu(x)) 
\end{pmatrix} ^{-1}
\begin{pmatrix}
g'(x)-h'(\xi(x))\\0
\end{pmatrix} \\
&=&\begin{pmatrix}
\frac{g'(x) - h'(\xi(x))}{(x-\xi(x))h''(\xi(x))}\\[2mm]
\frac{\mu(x)-\xi(x))(g'(x)-h'(\xi(x)))}{(x-\xi(x))(g'(\mu(x)) - h'(\xi(x)))}
\end{pmatrix}.
\end{eqnarray*}
The conversion to the parameter $t$ of the parametrization $\gamma:S^1\to\R^2$ does not
change this regularity. Indeed, if $\gamma(t_1)=(x,g(x))$ and $\gamma(t_2)=(\mu,g(\mu))$ 
we have $t_2=\gamma_1^{-1}(\mu(\gamma_1(t_1)))$.

With this regularity established the properties of the function $f$ follow from the properties of the Poncelet polygon in $(K,C)$.
\end{proof}

\begin{remark}
Notice that according to Proposition~\ref{pro-reg}, Theorem~\ref{poncelet-pair-no-self-intersection} describes the general case of a 
Poncelet pair $(K,C)$ if $K$ is a $C^2$ curve, and $C$ a $C^2$ curve with non-vanishing curvature.
Indeed, let $(K,C)$ be such a pair, where $K$ is given by a positively oriented parametrization,
then, by Proposition~\ref{pro-reg}, this defines an orientation preserving  $C^2$ diffeomorphism $f:S^1\to S^1$ with $f^n=\operatorname{id}_{S^1}$, i.e., a torsion map,
such that the curve  given by the parametrization~(\ref{envelopeC}) is the original envelope $C$.
\end{remark}

In Section~\ref{sec-ep} we have constructed equiangular Poncelet pairs for a given envelope $C$.
Now, we want to consider equiangular pairs, if the vertex curve $K$  is given.
Let $m,n\in\mathbb N$ be coprime, $n\geq 2$ and consider again a vertex curve $K$ with positive curvature given by the support function $p(\varphi) = a+\cos(\ell\varphi)$, where $\ell = \frac{n}{m}$. The condition for positive curvature is given by $a>\ell^2-1$ if $\ell>1$ and $a>1-\ell^2$ if $0<\ell<1$.
In this case $K$ can be parametrized by
\begin{equation}\label{eq-KKK}
K: \left[0,2\pi m\right)\to\R^2, \varphi\mapsto Y(\varphi) = p(\varphi)u(\varphi)+p'(\varphi)u'(\varphi).
\end{equation}

If $\alpha = 2\pi/\ell$, then the function $f(\varphi) = \varphi + \alpha \mod 2\pi$ is a smooth torsion map as $f^n = \mathrm{id}_{S^1}$ and $f^i\ne\mathrm{id}_{S^1}$ for $0<i<n$. 

The envelope of the segments of the Poncelet polygon $P$ with vertices
$$
Y(t+j\alpha), j=0,\ldots,n-1
$$
is given by
\begin{equation}\label{eq-CCC}
C:\left[0,2\pi m\right)\to\R^2, \varphi \mapsto X(\varphi) = Y(\varphi) +s(\varphi)(Y(\varphi+\alpha)-Y(\varphi)),\end{equation}
where
$$
s(\varphi) = \frac12\left(1+\cot^2\left(\frac\pi \ell\right)\frac{p'(\varphi)}{p(\varphi)}\right).
$$
Note that $s$ is well-defined, provided $a>1$ and in this case, the pair $(K,C)$ given by~\eqref{eq-KKK} and~\eqref{eq-CCC} is an equiangular Poncelet pair with angle $\alpha$.

Note that the polygon collapses to a multiplicity 2 segment if $n=2$ and in this case, the curve $C$ degenerates to a single point since $X(\varphi)=(0,0)$. We will therefore now assume $n\geq 3$:
If $(1 + a - 2 \ell^2)^2 > 0$, then the curvature of $X$ is positive and $X$ is therefore a regular curve provided $a>\max\{2\ell^2-1,1\}$.

If $\ell\in\N$ (and hence $\ell\geq 3$), $P$ is a regular $\ell$-gon and if $\ell$ is odd, the curve $Y$ is the boundary of a body of constant width as
$$
p(\varphi+\pi)+p(\varphi) \equiv 2a.
$$

Summarizing we have:
\begin{prop}\label{innercurves-iterated}
Let $m,n\in\mathbb N$ be coprime, $n\geq 3$,
$\ell=\frac{n}{m}$, $\alpha = \frac{2\pi}{\ell}$ and $p(\varphi) = a + \cos(\ell\varphi)$, where $a>\max\{\ell^2-1,1\}$. Then the Poncelet polygon $P$ for the Poncelet pair $(K,C)$ given by~\eqref{eq-KKK} and~\eqref{eq-CCC} is equiangular with angle $\alpha$ and $K$ is a curve with positive curvature. We furthermore have:
\begin{enumerate}\item If $a>\max\{2\ell^2-1,1\}$, the curvature of $C$ is positive.
\item If $\ell\in\N$, then $P$ is a regular $\ell$-gon and the condition for positive curvature of $C$ reduces to $a>2\ell^2-1$.
\item If $\ell\in\N$ is odd, then $K$ is the boundary of a body with constant width.
\end{enumerate}
\end{prop}

\begin{remark}
If $s(\varphi) \in[0,1]$, then the segments of $P$ rather than their prolongation touch $C$. This condition translates into
$\left|\cot\left(\frac\pi \ell\right)p'(\varphi)/p(\varphi)\right|\leq 1$. Since $|p'/p|\leq \frac{\ell}{\sqrt{a^2-1}}$, it is sufficient to require that
\begin{equation}\label{req-touch}
a\geq \sqrt{\cot^2\left(\frac{\pi}{\ell}\right)\ell^2+1}.
\end{equation}
If $\ell\in\mathbb N$, and the condition $a>2\ell^2-1$ for positive curvature of $C$ is satisfied, then~\eqref{req-touch} is automatically met since in this case
$2\ell^2-1>\sqrt{\cot^2\left(\frac{\pi}{\ell}\right)\ell^2+1}.$
\end{remark}

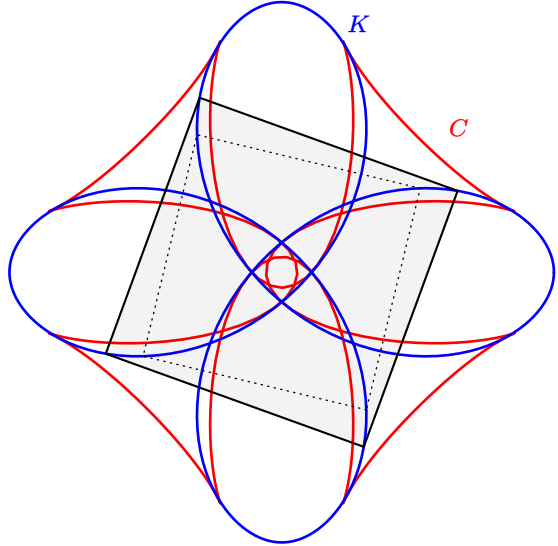
\begin{figure}[h]
\begin{center}
\begin{tikzpicture}[line cap=round,line join=round,x=7.8,y=7.8]
\draw [line width=1pt,domain=0:18.8496,samples=500,red] 
plot({6/(1.17778 + cos(4*(\x r)/3))*(1.37407*cos((\x r)/3) + 0.499136*cos((\x r)) + 0.680556*cos((5*(\x r))/3) - 0.196296*cos((7*(\x r))/3) + 
   0.0138889*cos((11*(\x r))/3) - 1.37407*sin((\x r)/3) + 0.499136*sin((\x r)) - 
   0.680556*sin((5*(\x r))/3) - 0.196296*sin((7*(\x r))/3) + 
   0.0138889*sin((11*(\x r))/3))},
{6/(1.17778 + cos((4*(\x r))/3))*(-1.37407*cos((\x r)/3) - 0.499136*cos(\x r) - 0.680556*cos((5*(\x r))/3) + 0.196296*cos((7*(\x r))/3) - 
   0.0138889*cos((11*(\x r))/3) - 1.37407*sin((\x r)/3) + 0.499136*sin(\x r) - 
   0.680556*sin((5*(\x r))/3) - 0.196296*sin((7*(\x r))/3) + 
   0.0138889*sin((11*(\x r))/3))}
   );

\draw [line width=1pt,domain=0:18.8496,samples=300,blue] plot(
{6*(1.16667*cos((\x r)/3) + 1.17778*cos(\x r) - 0.166667*cos((7* (\x r))/3))}, {6*(-1.16667*sin((\x r)/3) + 1.17778*sin(\x r) - 0.166667*sin((7* (\x r))/3))}
);

\tikzset{declare function={f1(\j)=6*(1.16667*(cos(((\x)/3+1.5708*\j) r) + 1.00952*cos((\x+4.71239*\j) r) - 0.142857*cos(((7* (\x))/3+10.9956*\j) r)));}};
\tikzset{declare function={f2(\j)=6*(-1.16667*(sin(((\x )/3+1.5708*\j) r) - 1.00952*sin((\x +4.71239*\j) r) + 0.142857*sin(((7* (\x ))/3+10.9956*\j) r)));}};
\def\x{1.4}
\draw [line width=.8pt,fill=black, fill opacity=.05]({f1(1)},{f2(1)}) -- ({f1(2)},{f2(2)})--({f1(3)},{f2(3)})--({f1(4)},{f2(4)})--cycle;
\def\x{1.6}
\draw[line width=.5pt,dotted] ({f1(1)},{f2(1)}) -- ({f1(2)},{f2(2)})--({f1(3)},{f2(3)})--({f1(4)},{f2(4)})--cycle;
\begin{footnotesize}
\draw[color=blue] (3.7,12) node {$K$};
\draw[color=red] (8.5,7) node {$C$};
\end{footnotesize}
\end{tikzpicture}
\caption{The curves $K$ and $C$ obtained in Proposition~\ref{innercurves-iterated} with $m=3$, $n=4$ and $a=\frac{53}{45}$.
Note that here the extensions of the sides of the Poncelet square touch the envelope $C$.}
\label{figure-no-touch}
\end{center}
\end{figure}
\subsection{Poncelet clans for given vertex curve \boldmath$K$}\label{sec-3.3}

Let again $K:S^1\to\R^2, \varphi \mapsto Y(\varphi)$, be a positively oriented regular $C^2$ curve with non-vanishing curvature. This time we consider orientation preserving diffeomorphisms $f_i:S^1\to S^1$, $1\leq i\leq n$, where $n\geq 3$ such that
$$
f_n=(f_{n-1}\circ \ldots \circ f_1)^{-1}.
$$
We will henceforth write $g_i = f_i\circ f_{i-1}\circ\ldots\circ f_1$ for $1\leq i\leq n$ and define $g_0=\mathrm{id}_{S^1}$.
For the following construction, we require that the Poncelet polygon with vertices $Y(g_i(\varphi)), i\in\{0,\ldots,n-1\}$, is not degenerate for all $\varphi\in S^1$.
By this we mean that the vertices are always  different from each other.
This condition can be reformulated as follows:
If two vertices of the polygon coincide, then we must have $g_j(\varphi_0)=g_i(\varphi_0)$ for some $\varphi_0\in S^1$ and $i\ne j$. We may assume without loss of generality that $i<j$ so that $g_i(\varphi_0)$ must be a fixed point of $g_j\circ g_i^{-1}$. Therefore we will require $g_j\circ g_i^{-1}$ to not have fixed points for all $1\leq i<j\leq n$.
The curve  $C_i:S^1\to\R^2$ is then given by the envelope of the segments
$(1-s)Y(g_{i-1}(\varphi))+sY(g_i(\varphi))$, where $i$ is taken mod $n$. According to \eqref{envelopeC}, $C_i$ will be parametrized by
$$
X_i = Y\circ g_{i-1}-\frac{\langle(Y\circ g_{i-1})',J(Y\circ g_i-Y\circ g_{i-1})\rangle}{\langle(Y\circ g_i-Y\circ g_{i-1})',J(Y\circ g_i-Y\circ g_{i-1})\rangle}(Y\circ g_i-Y\circ g_{i-1}),
$$
and $(K,C_1,\ldots,C_n)$ is a Poncelet clan.

\subsection{Poncelet pairs for given envelope \boldmath$C$}\label{sec-3.4}
Here, we use the same approach as in Section~\ref{sec-2}. Let the envelope $C$ be given by~(\ref{eq-C}).
Consider an orientation preserving diffeomorphism $f \in C^2(S^1_k,S^1_k)$, where $S^1_k = \R/2k\pi \Z$ such that $f^0 =f^n =
\operatorname{id}_{S^1_k}$, but $f^i\neq \operatorname{id}_{S^1_k}$ for $0<i<n$ with $n>2$. 
Then consider $X(\varphi)$ and $X(f(\varphi))$ as points of contact of two consecutive sides of a Poncelet polygon with
the given curve $C$. The geometric situation is as in Figure~\ref{fig-2}, but  with $f(\varphi)$ in place of $\varphi+\alpha$.
Then the point $Y(\varphi) = p(\varphi)u(\varphi) + q(\varphi)u'(\varphi)$ on the vertex curve $K$ satisfies
\begin{equation}\label{eq-euleralg}
p(\varphi) u(\varphi) + 
 q(\varphi) u'(\varphi) =p(f(\varphi)) u(f(\varphi)) + 
   s(\varphi) u'(f(\varphi)).
\end{equation}
As in Section~\ref{sec-2} we find
$$
q(\varphi)= \frac{p(f(\varphi))-p(\varphi)\langle u(\varphi),u(f(\varphi))\rangle}{\langle u'(\varphi),u(f(\varphi))\rangle}.
$$
Hence we have the following result.
\begin{theorem}\label{thm18}
Let $C$ be a closed $C^2$ curve in the Euclidean plane with non-vanishing curvature, given by~(\ref{eq-C}),
and $f \in C^2(S^1_k,S^1_k)$ an orientation preserving diffeomorphism  such that $f^0 =f^n =
\operatorname{id}_{S^1_k}$, but $f^i\neq \operatorname{id}_{S^1_k}$ for $0<i<n$ with $n>2$. 
Suppose that $\langle u'(\varphi),u(f(\varphi))\rangle\neq 0$ for all $\varphi$. Then $(K,C)$
is a Poncelet pair for the vertex curve 
$$K:[0,2k\pi)\to\mathbb R^2, \quad \varphi\mapsto Y(\varphi)=p(\varphi)u(\varphi)+\frac{p(f(\varphi))-p(\varphi)\langle u(\varphi),u(f(\varphi))\rangle}{\langle u'(\varphi),u(f(\varphi))\rangle} u'(\varphi).
$$
\end{theorem}

\subsection{Poncelet clans for given envelope \boldmath$C$}\label{sec-3.5}

Let $C:S^1_k\to \R^2$ be a positively oriented regular $C^2$ curve with non-vanishing curvature given by a support function $p$ as $X(\varphi) = p(\varphi)u(\varphi)+p'(\varphi)u'(\varphi)$ and consider nontrivial orientation preserving $C^2$ diffeomorphisms $f_i:S^1_k\to S^1_k$, $1\leqslant i \leqslant n$, where $n\geqslant 3$ such that
$$
f_n = (f_{n-1}\circ \ldots \circ f_1)^{-1}
$$
and we will write again $g_i = f_i\circ f_{i-1}\circ \ldots \circ f_1$ for $1\leqslant i \leqslant n$ and define $g_0=\mathrm{id}_{S^1_k}$. Imitating the construction given in Theorem~\ref{thm18}, if $\langle u'(g_{i-1}(\varphi)),u(g_{i}(\varphi)\rangle \ne 0$ for all $\varphi$ and $i\in\{1,\ldots,n\}$, we obtain for all $i\in\{1,\ldots,n\}$ a $C^2$ vertex curve $K_i:S^1_k\to\mathbb R^2$ defined by
$$
\varphi\mapsto Y_i(\varphi)=p(g_{i-1}(\varphi))u(g_{i-1}(\varphi))+\frac{p(g_i(\varphi))-p(g_{i-1}(\varphi))\langle u(g_{i-1}(\varphi)),u(g_i(\varphi))\rangle}{\langle u'(g_{i-1}(\varphi)),u(g_i(\varphi))\rangle} u'(g_{i-1}(\varphi)).
$$
The polygon $P(\varphi)$ with vertices $Y_1(\varphi),Y_2(\varphi),\ldots, Y_n(\varphi)$ is then a Poncelet polygon for each value of $\varphi$.

\subsection*{Statements and Declarations}
The authors report there are no competing interests to declare.

\subsection*{Author contributions}
Both authors contributed equally to the article.

\section*{Acknowledgement}
We would like to thank the referees for their valuable remarks which helped to improve this article.

\bibliographystyle{plain}

\begin{thebibliography}{10}

\bibitem{ahw}
Jonas Allemann, Norbert Hungerb\"{u}hler, and Micha Wasem.
\newblock Equilibria of plane convex bodies.
\newblock {\em J. Nonlinear Sci.}, 31(5):Paper No. 86, 19, 2021.

\bibitem{gorkin2020}
Kelly Bickel and Pamela Gorkin.
\newblock Numerical range and compressions of the shift.
\newblock In {\em Complex analysis and spectral theory}, volume 743 of {\em
  Contemp. Math.}, pages 241--261. Amer. Math. Soc., [Providence], RI, [2020]
  \copyright 2020.

\bibitem{gorkin}
Isabelle Chalendar, Pamela Gorkin, and Jonathan~R. Partington.
\newblock Inner functions and operator theory.
\newblock {\em North-West. Eur. J. Math.}, 1:7--22, 2015.

\bibitem{chien}
Mao-Ting Chien and Hiroshi Nakazato.
\newblock A new {P}oncelet curve for the boundary generating curve of a
  numerical range.
\newblock {\em Linear Algebra Appl.}, 487:1--21, 2015.

\bibitem{blaschke}
Ulrich Daepp, Pamela Gorkin, Andrew Shaffer, and Karl Voss.
\newblock {\em Finding ellipses}, volume~34 of {\em Carus Mathematical
  Monographs}.
\newblock MAA Press, Providence, RI, 2018.
\newblock What Blaschke products, Poncelet's theorem, and the numerical range
  know about each other.

\bibitem{vladimir}
Vladimir Dragovi\'{c} and Milena Radnovi\'{c}.
\newblock {\em Poncelet porisms and beyond}.
\newblock Frontiers in Mathematics. Birkh\"{a}user/Springer Basel AG, Basel,
  2011.
\newblock Integrable billiards, hyperelliptic Jacobians and pencils of
  quadrics.

\bibitem{flatto}
Leopold Flatto.
\newblock {\em Poncelet's theorem}.
\newblock American Mathematical Society, Providence, RI, 2009.
\newblock Chapter 15 by S. Tabachnikov.

\bibitem{gau}
Hwa-Long Gau and Pei~Yuan Wu.
\newblock Numerical range of {$S(\phi)$}.
\newblock {\em Linear and Multilinear Algebra}, 45(1):49--73, 1998.

\bibitem{wu-gau-survey}
Hwa-Long Gau and Pei~Yuan Wu.
\newblock Numerical range and {P}oncelet property.
\newblock {\em Taiwanese J. Math.}, 7(2):173--193, 2003.

\bibitem{glaeser}
Georg Glaeser, Hellmuth Stachel, and Boris Odehnal.
\newblock {\em The Universe of Conics: From the ancient Greeks to 21st century
  developments}.
\newblock Springer, Berlin, 1st edition, 2016.

\bibitem{hhponcelet}
Lorenz Halbeisen and Norbert Hungerb\"{u}hler.
\newblock A simple proof of {P}oncelet's theorem (on the occasion of its
  bicentennial).
\newblock {\em Amer. Math. Monthly}, 122(6):537--551, 2015.

\bibitem{hh}
Lorenz Halbeisen and Norbert Hungerb\"{u}hler.
\newblock The exponential pencil of conics.
\newblock {\em Beitr. Algebra Geom.}, 59(3):549--571, 2018.

\bibitem{hunziker}
Markus Hunziker, Andrei Mart\'inez-Finkelshtein, Taylor Poe, and Brian Simanek.
\newblock {Poncelet-{D}arboux, {K}ippenhahn, and {S}zeg\H o: interactions
  between projective geometry, matrices and orthogonal polynomials}.
\newblock {\em J. Math. Anal. Appl.}, 511(1):Paper No. 126049, 35, 2022.

\bibitem{king}
Jonathan~L. King.
\newblock Three problems in search of a measure.
\newblock {\em The American Mathematical Monthly}, 101(7):609--628, 1994.

\bibitem{kolodziej}
Rafa{\l} Ko{\l}odziej.
\newblock The rotation number of some transformation related to billiards in an
  ellipse.
\newblock {\em Studia Math.}, 81(3):293--302, 1985.

\bibitem{lopes}
Artur~O. Lopes and Marcos Sebastiani.
\newblock Poncelet pairs and the twist map associated to the {P}oncelet
  billiard.
\newblock {\em Real Anal. Exchange}, 35(2):355--374, 2010.

\bibitem{mirman98}
Boris Mirman.
\newblock Numerical ranges and {P}oncelet curves.
\newblock {\em Linear Algebra Appl.}, 281(1-3):59--85, 1998.

\bibitem{mirman2009}
Boris Mirman.
\newblock Poncelet's porism in the finite real plane.
\newblock {\em Linear Multilinear Algebra}, 57(5):439--458, 2009.

\bibitem{mirman2001}
Boris Mirman, Vladimir Borovikov, Lev Ladyzhensky, and Robert Vinograd.
\newblock Numerical ranges, {P}oncelet curves, invariant measures.
\newblock {\em Linear Algebra Appl.}, 329(1-3):61--75, 2001.

\bibitem{poincare}
Henri Poincar\'e.
\newblock Sur les courbes définies par les équations différentielles (iii).
\newblock {\em J. Math. Pures Appl.}, 1:167--244, 1885.

\bibitem{Sanhueza2018}
Diego A.~S. Sanhueza.
\newblock Lecture notes on dynamical systems: Homeomorphisms on the circle,
  2018.
\newblock Accessed via DocsLib.

\bibitem{turer}
George Turer.
\newblock Dynamical systems on the circle.
\newblock In {\em REUPapers}. The University of Chicago, 2019.
\newblock https://api.semanticscholar.org/CorpusID:207991365.

\bibitem{wu-gau}
Pei~Yuan Wu and Hwa-Long Gau.
\newblock {\em Numerical ranges of {H}ilbert space operators}, volume 179 of
  {\em Encyclopedia of Mathematics and its Applications}.
\newblock Cambridge University Press, Cambridge, 2021.

\end{thebibliography}

\end{document}